\documentclass{article}
\usepackage{amsfonts}
\usepackage{amsmath}
\usepackage{graphicx}
\usepackage{mathtools}

\newtheorem{theorem}{Theorem}

\newtheorem{conjecture}[theorem]{Conjecture}

\newtheorem{proposition}[theorem]{Proposition}
\newtheorem{remark}[theorem]{Remark}

\newenvironment{proof}[1][Proof]{\noindent\textbf{#1.} }{\ \rule{0.5em}{0.5em}}
\input{tcilatex}
\allowdisplaybreaks

\begin{document}

\title{Fractional Cointegration of Geometric Functionals}
\author{Alessia Caponera\thanks{%
Department of AI, Data and Decision Sciences, LUISS, email:acaponera@luiss.it%
}, Domenico Marinucci\thanks{%
Department of Mathematics, University of Rome Tor Vergata,
email:marinucc@mat.uniroma2.it}, Anna Vidotto\thanks{%
Department SBAI, Sapienza University of Rome, email: anna.vidotto@uniroma1.it%
}}
\maketitle

\begin{abstract}
In this paper, we show that geometric functionals (e.g., excursion area,
boundary length) evaluated on excursion sets of sphere-cross-time long
memory random fields can exhibit fractional cointegration, meaning that some
of their linear combinations have shorter memory than the original vector.
These results prove the existence of long-run equilibrium relationships
between functionals evaluated at different threshold values; as a
statistical application, we discuss a frequency-domain estimator for the
Adler-Taylor metric factor, i.e., the variance of the field's gradient. Our
results are illustrated also by Monte Carlo simulations.

\begin{itemize}
\item Keywords and Phrases: Sphere-cross-time Random Fields, Geometric
Functionals, Long Memory, Fractional Cointegration

\item AMS Classification: 60G60; 62M10, 62M15, 62M40
\end{itemize}
\end{abstract}

\tableofcontents

\section{Introduction}

\subsection{Sphere-cross-time random fields and geometric functionals}

We consider a centered, unit-variance, Gaussian isotropic and stationary
sphere-cross-time random field, that is a collection $\{Z(x,t),\ (x,t)\in
\mathbb{S}^{2}\times \mathbb{Z}\}$ of jointly Gaussian random variables with
unit variance and such that
\begin{eqnarray*}
\mathbb{E}[Z(x,t)] &=&0, \\
\mathbb{E}[Z(x_{1},t_{1})Z(x_{2},t_{2})]&=:&\Gamma (\langle
x_{1},x_{2}\rangle ,t_{2}-t_{1}),
\end{eqnarray*}%
with $\langle \cdot ,\cdot \rangle $ being the standard inner product in $%
\mathbb{R}^{3}$. It is well-known (see, e.g., \cite{MaPeCUP}) that the space
of square-integrable functions on the sphere admits as an orthonormal basis
the fully-normalized spherical harmonics
\begin{equation*}
\left\{ Y_{\ell m},\ m=-\ell ,\dots ,\ell ,\ \ell =0,1,2,\dots \right\}
\text{ },
\end{equation*}
from which it is possible to obtain the following spectral representation
\begin{equation}
Z(x,t)=\sum_{\ell =0}^{\infty }\sum_{m=-\ell }^{\ell }a_{\ell m}(t)Y_{\ell
m}(x)\,,  \label{specrap}
\end{equation}
where $\left\{ a_{\ell m}(t):=\langle Z,Y_{\ell m}\rangle _{L^{2}(\mathbb{S}%
^{2})}\right\} _{\ell =0,1,...,m=-\ell ,...,\ell }$ are the so called
(random) spherical harmonics coefficients; under isotropy, for any given $%
\ell $ the processes $a_{\ell m}(.)$ are independent and have the same law,
for all $m=-\ell ,...,\ell $. The representation \ref{specrap} is often
called the multipole expansion of the field $Z$, where $Z_{\ell
}(x,t):=\sum_{m=-\ell }^{\ell }a_{\ell m}(t)Y_{\ell m}(x)$ is the $\ell $-th
order multipole of $Z$; however, with an abuse of language, the index $\ell $
is also called a multipole when no confusion is possible.

We assume that there exists at least a $\ell $ such that the stochastic
processes $a_{\ell m}(.)$ are long memory, meaning that their covariance
function $E[a_{\ell m}(t_{2})a_{\ell m}(t_{1})]=C_{\ell }(t_{2}-t_{1})$ is
such that%
\begin{equation*}
\lim_{\tau \rightarrow \infty }\frac{C_{\ell }(\tau )}{C_{\ell }(0)(1+\tau
)^{2d_{\ell }-1}}=const>0\,,\quad 0<d_{\ell }<\frac{1}{2}\,.
\end{equation*}%
In particular, note that $d_{\ell }>d_{\ell ^{\prime }}$ implies that the
multipole $\ell $ has ``longer memory'' than the multipole $\ell ^{\prime }$%
- i.e., covariances that decay at a slower rate. As a consequence, we will
assume that there exists
\begin{equation*}
d_{\ast }:=\max \{d_{\ell }:\ell \in \widetilde{\mathbb{N}},\ell \geq
1\},\quad \text{ with }\quad \widetilde{\mathbb{N}}:=\{\ell \in \mathbb{N}%
:C_{\ell }(0)\neq 0\}\,,
\end{equation*}
and we are going to denote by $d_{\ast \ast }:=\max \{d_{\ell }:\ell \in
\mathbb{N}\setminus \mathcal{I}^{\ast }\}$ and by $\widetilde{d}%
_{\ast}:=d_{\ast } \vee d_{0}$. Moreover, the set $\mathcal{I}^{\ast
}:=\{\ell \in \mathbb{N}:d_{\ell }=\widetilde{d}_{\ast }\}$ represent the
multipoles of the field $Z$ with the longest memory (clearly, $\widetilde{d}%
_{\ast }\in \left( 0,\frac{1}{2}\right) $). Essentially, the covariance
decay with respect to time of $Z(\cdot,\cdot)$, as the lag $\tau \rightarrow
\infty $, is completely determined by these subsets of harmonic components:
\begin{equation*}
\sum_{\ell \in \mathcal{I}^{\ast }}^{\infty }Z_{\ell }(x,t)=\sum_{\ell \in
\mathcal{I}^{\ast }}^{\infty }\sum_{m=-\ell }^{\ell }a_{\ell m}(t)Y_{\ell
m}(x).
\end{equation*}

We are interested in the following local geometric functionals, that are
just the area and the boundary length of the $u$-level excursion sets of the
random field $Z(x,t)$, for $u\in \mathbb{R},$
\begin{align*}
A(u;t) &=\text{area}(Z(x,t)> u)=\int_{\mathbb{S}^2}\mathbf{1}%
_{\{Z(x,t)>u\}}dx\,, \\
\mathcal{L}(u;t) &=\text{length}(Z(x,t)=u)=\int_{\mathbb{S}%
^2}\delta_u(Z(x,t)) \|\nabla Z(x,t)\|dx\,.
\end{align*}

In \cite{AAP2021,AHL2024}, the following Wiener chaos expansions for $A(u;t)$
and $\mathcal{L}(u;t)$ have been established:
\begin{equation*}
A(u;t)-\mathbb{E}[A(u;t)]=\sum_{q=1}^{\infty }\frac{H_{q-1}(u)\phi (u)}{q!}%
\int_{\mathbb{S}^{2}}H_{q}(Z(x,t))dx
\end{equation*}%
\begin{align}
& =\frac{1}{\sqrt{4\pi }}\phi (u)a_{00}(t)+u\phi (u)\sum_{\ell }C_{\ell
}(2\ell +1)\left\{ \widehat{C}_{\ell }(t)-C_{\ell }\right\}  \notag \\
& \quad +\sum_{q=3}^{\infty }\frac{H_{q-1}(u)\phi (u)}{q!}\int_{\mathbb{S}%
^{2}}H_{q}(Z(x,t))dx  \label{eq:chaotic-exp-A}
\end{align}%
and%
\begin{align}
&\mathcal{L}(u;t)-\mathbb{E}[\mathcal{L}(u;t)] =\sigma _{1}\sqrt{2}\pi u\phi
(u)a_{00}(t)  \label{eq:chaotic-exp-L} \\
&+\frac{\sigma _{1}}{2}\sqrt{\frac{\pi }{2}}\phi (u)\sum_{\ell }(2\ell
+1)\left\{ (u^{2}-1)+\frac{\lambda _{\ell }/2}{\sigma _{1}^{2}}\right\}
\left\{ \widehat{C}_{\ell }(t)-C_{\ell }\right\} +\sum_{q=3}^{\infty }%
\mathcal{L}(u,t)[q]\,,  \notag
\end{align}%
where $\lambda _{\ell }=\ell (\ell +1)$ is (minus) the $\ell $-th eigenvalue
of the spherical Laplacian and
\begin{equation*}
\widehat{C}_{\ell }(t):=\sum_{m=-\ell }^{\ell }\frac{a_{\ell ,m}(t)^{2}}{%
2\ell +1}\,\text{\ ;}
\end{equation*}%
here we denoted by $\mathcal{L}(u,t)[q]$ the $q$-th order chaotic component
of the length functional; we will denote analogously by $A(u,t)[q]$ the $q$%
-th order chaotic component of the area functional. Note that the terms that
involve $\widehat{C}_{\ell }(t)$ are the second order chaotic components of
the two functionals, that is
\begin{align*}
A(u,t)[2]& =u\phi (u)\sum_{\ell }C_{\ell }(2\ell +1)\left\{ \widehat{C}%
_{\ell }(t)-C_{\ell }\right\} \\
\mathcal{L}(u,t)[2]& =\frac{\sigma _{1}}{2}\sqrt{\frac{\pi }{2}}\phi
(u)\sum_{\ell }(2\ell +1)\left\{ (u^{2}-1)+\frac{\lambda _{\ell }/2}{\sigma
_{1}^{2}}\right\} \left\{ \widehat{C}_{\ell }(t)-C_{\ell }\right\} \,,
\end{align*}%
while the terms that involve $a_{00}(t)$ are the first order chaotic
components of the two functionals; see \cite{AHL2024} for analytic
expressions of $\mathcal{L}(u,t)[q],$ $q\geq 3$. It is important to stress
that as we took time to be discrete and thus countable, we are avoiding any
measurability issue for these functionals.

Time dependent random fields on the spheres (or on other manifolds) have
been the object of quite a few papers recently, see for instance \cite%
{AAP2021}, \cite{AHL2024}, \cite{AoS2001}, \cite{Caponera2024}, \cite%
{CaponeraLasso},\cite{CaponeraSpharma}\cite{LeonenkoMedina},\cite{MalPorcu},%
\cite{Ovalle},\cite{Porcu2023} and the references therein; for geometric
functionals, we refer to the classical textbooks by \cite{AdlerTaylor},\cite%
{AzaisWschebor}.\newline

We need now to recall the notion of fractional cointegration from the
statistical literature. A $p\times 1$ stationary process $X(t)$, possibly
vector-valued (i.e., $p>1$), is \emph{fractionally integrated of order} $d$,
writing $X(t)\sim I(d)$, if it is long memory with autocorrelation function $%
\rho$ satisfying
\begin{equation*}
\lim_{\tau \rightarrow \infty }\frac{\rho(\tau)}{(1+\tau)^{2d-1}}=\text{%
constant}>0\,,\quad 0<d<\frac{1}{2}\,.
\end{equation*}
Moreover, we say that $X(t)\sim I(d)$ is \emph{fractionally cointegrated of
order} $b$, written $CI(d,b),$ if there exist a $p\times 1$ vector $\gamma $
such that $\gamma ^{\prime }X(t)\sim I(d-b),$ $0<b\leq d.$ We will say that
the vector $X(t)$ is \emph{multi cointegrated} if there exist multiple
vectors $\gamma _{1},...,\gamma _{m}$ such that $\gamma _{i}^{\prime
}X(t)\sim I(d-b_{i}),$ $0<b_{i}\leq d,$ $m<p.$ We call $m$ the cointegrating
rank of $X(t).$

As mentioned earlier, the notion of fractional cointegration has emerged in
the time series literature in the late nineties-early 2000, as a natural
evolution of the very successful (integer-valued) cointegration paradigm
which was introduced in the time series econometrics literature form the
eighties. Although the possibility of fractional cointegration was hinted
already in some previous works, the first paper where a rigorous attempt to
a definition and statistical analysis was given is \cite{Robinson1994}; some
further developments were given a few years afterwards in \cite{RM2001},
\cite{MR2001}, \cite{RM2003}, \cite{ChenHurvich2003}, \cite{ChenHurvich2006}
and then by a very rich literature, see for instance \cite{Iaco2019}, \cite%
{Johansen2019}, \cite{Robinson2019} and the references therein for more
recent contributions.\newline

Our purpose in this paper is to show how the same concept emerges naturally
(and surprisingly) in the analysis of some geometric functionals for
sphere-cross-time random fields with long range dependence.

N\noindent \textbf{otation}. From now on, $c\in (0,+\infty )$ will stand for
a universal constant which may change from line to line. Let $\{a_{n},n\geq
0\}$, $\{b_{n},n\geq 0\}$ be two sequences of positive numbers: we will
write $a_{n}\simeq b_{n}$ if $a_{n}/b_{n}\rightarrow 1$ as $n\rightarrow
+\infty $, $a_{n}\approx b_{n}$ whenever $a_{n}/b_{n}\rightarrow c>0$, $%
a_{n}=o(b_{n})$ if $a_{n}/b_{n}\rightarrow 0$, and finally $a_{n}=O(b_{n})$
if eventually $a_{n}/b_{n}\leq c$.

\subsection{Memory of the field \emph{vs} memory of the functionals}

\label{sec:memory}

Before going into the main results of our work, let us discuss briefly the
relationship between the dependence properties of the primitive random field
$Z=\{Z(x,t),\ (x,t)\in \mathbb{S}^{2}\times \mathbb{Z}\}$ and the memory of
the geometric functionals considered here. Recall that the covariance decay
of $Z$ is completely determined by the multipoles $Z_{\ell }$ such that $%
d_{\ell }=\widetilde d_{\ast }$, that is
\begin{equation*}
Z_{\ast }(x,t)=\sum_{\ell: d_{\ell}=\widetilde d_{\ast }}Z_{\ell }(x,t)
\end{equation*}%
and hence its order of integrations is simply given by $\widetilde d_{\ast }$%
, that is
\begin{equation*}
Z(x,t)\sim I\left(\widetilde d_{\ast }\right)\sim Z_{\ast }(x,t)\,.
\end{equation*}
On the other hand, we have that our geometric functionals can be written as
a sum of their chaotic components, as we have seen in %
\eqref{eq:chaotic-exp-A} and \eqref{eq:chaotic-exp-L}. Hence, in order to
establish the memory of $A(u,t)$ and $\mathcal{L}(u,t)$ we investigate the
covariance decay of their chaotic components. Both the first chaoses of $%
A(u,t)$ and $\mathcal{L}(u,t)$ are given by some constants (involving the
level $u$) times the zeroth Fourier components of the original field $Z(x,t)$%
, that is $a_{00}(t)$; \ as a consequence, it is immediate to get that
\begin{equation*}
A(u,t)[1]\,,\,\mathcal{L}(u,t)[1]\sim I(d_{0})\,.
\end{equation*}
For the $q$-th order chaotic component, $q\geq 2$, which involves integrals
of $q$-th order Hermite polynomials evaluated at $Z(x,t)$, it holds
\begin{equation}
A(u,t)[q]\,,\,\mathcal{L}(u,t)[q]\sim I\left( q\,\widetilde{d}_{\ast }-\frac{%
q-1}{2}\right) \,.  \label{int-order-q}
\end{equation}%
This is a consequence of the so-called \emph{diagram formula} (see \cite[%
Proposition 4.15]{MaPeCUP}), which relates the moments of Hermite
polynomials of Gaussian random variables with powers of their covariances;
in our case we obtain:
\begin{equation*}
\mathbb{E}\left[ H_{q}(Z(x,t))H_{q}(Z(y,s))\right] =q!\,\mathbb{E}\left[
Z(x,t)Z(y,s)\right] ^{q}\,.
\end{equation*}%
Since, for any fixed $d \in (0,1/2)$, $q\,d-(q-1)/{2}$ is a decreasing
function in $q$, the order of integration of $A(u,t)$ and $\mathcal{L}(u,t)$
is determined either by the relationship between $d_{0}$ and $2\,d_{\ast
}-1/2$ (in the cases in which the second order chaotic component do not
vanish) or by the one between $d_{0}$ and $3d_{\ast }-1$. We are neglecting
for brevity a further case, namely the one of asymptotically monochromatic
random fields (see \cite{AHL2024} and below) where a cancellation at some
specific threshold level $u^{\ast }$ introduces two further alternatives for
the boundary length. \newline

We are now in the position to establish our main results.

\section{Main Results}

\subsection{Fractional cointegration of geometric functionals}

Let us define the vector processes%
\begin{eqnarray*}
A(u_{1},...,u_{p_{1}};t) &=&(A(u_{1};t),...,A(u_{p_{1}};t))^{\prime
}\,,\quad u_{1},...,u_{p_{1}}\in \mathbb{R}\,,\quad \\
\mathcal{L(}v_{1},...,v_{p_{2}};t\mathcal{)} &\mathcal{=}&\mathcal{(L(}%
v_{1};t\mathcal{)},...,\mathcal{L(}v_{p_{2}};t\mathcal{))}^{\prime }\, ,
\quad v_{1},...,v_{p_{2}}\in \mathbb{R}\,.
\end{eqnarray*}

Let us now introduce some more notation, namely the ($p_{1}-1)\times p_{1}$
matrices%
\begin{equation*}
\Gamma _{1}(u_{1},...,u_{p_{1}}):=\left(
\begin{array}{ccccc}
1 & -\frac{\phi (u_{1})}{\phi (u_{2})} & 0 & ... & 0 \\
1 & 0 & -\frac{\phi (u_{1})}{\phi (u_{3})} & ... & 0 \\
... & .... & ... & ... & ... \\
1 & 0 & ... & ... & -\frac{\phi (u_{1})}{\phi (u_{p_{1}})}%
\end{array}%
\right) \,\text{\ },
\end{equation*}%
and%
\begin{equation*}
\widetilde{\Gamma }_{1}(u_{1},...,u_{p_{1}})=\left(
\begin{array}{ccccc}
1 & -\frac{u_{1}\phi (u_{1})}{u_{2}\phi (u_{2})} & 0 & ... & 0 \\
1 & 0 & -\frac{u_{1}\phi (u_{1})}{u_{3}\phi (u_{3})} & ... & 0 \\
... & .... & ... & ... & ... \\
1 & 0 & ... & ... & -\frac{u_{1}\phi (u_{1})}{u_{p_{1}}\phi (u_{p_{1}})}%
\end{array}%
\right) \,.
\end{equation*}%
Our first results are the following

\begin{proposition}
\label{prop1} For $u_{1},...,u_{p_{1}}\in \mathbb{R}\backslash 0,$ and $%
d_{0}>2d_{\ast}-\frac{1}{2},$ the vector process $A(u_{1},...,u_{p_{1}};t)$
is fractionally cointegrated of rank $p_{1}-1$ and order:
\begin{equation*}
CI(d_0,d_0-(2d_{\ast}-{1}/{2})).
\end{equation*}

For $u_{1},...,u_{p_{1}}\in \mathbb{R}\backslash 0,$ and $d_{0}<2d_{\ast}-%
\frac{1}{2},$ the vector process $A(u_{1},...,u_{p_{1}};t)$ is fractionally
cointegrated of rank $p_{1}-1$ and order:
\begin{equation*}
CI(2d_{\ast}-{1}/{2},(2d_{\ast}-{1}/{2})-[d_0\vee (3d_\ast-1)]).
\end{equation*}

The cointegrating space is spanned by the vectors given by the rows of $\
\Gamma _{1}(u_{1},...,u_{p_{1}})$ and $\widetilde{\Gamma }%
_{1}(u_{1},...,u_{p_{1}}),$ respectively in the first and second case.
\end{proposition}

\begin{proof}
First assume that $d_{0}>2d_{\ell^{\ast }}-\frac{1}{2}$. Take the first row
of $\Gamma_1$, that is the $1\times p_1$ vector
\begin{equation*}
\gamma_1=\gamma_1(u_{1},u_{2}):=\left(%
\begin{array}{ccccc}
1 & -\frac{\phi (u_{1})}{\phi (u_{2})} & 0 & \cdots & 0%
\end{array}%
\right)\,,
\end{equation*}
then
\begin{align*}
\gamma_1 A(u_{1},...,u_{p_{1}};t)&= A(u_{1},t)-\frac{\phi (u_{1})}{\phi
(u_{2})}A(u_{2},t)\,.
\end{align*}
Exploiting the chaos expansion of the geometric functional $A(u;t)$, that is %
\eqref{eq:chaotic-exp-A}, we have
\begin{align*}
&A(u_{1},t)-\frac{\phi (u_{1})}{\phi (u_{2})}A(u_{2},t) \\
&=\mathbb{E}[A(u_{1},t)]+\frac{1}{\sqrt{4\pi }}\phi (u_1)a_{00}(t)
+\sum_{q=2}^{\infty }A(u_1,t;q) \\
&\quad-\frac{\phi (u_{1})}{\phi (u_{2})} \left[\mathbb{E}[A(u_{2},t)]+\frac{1%
}{\sqrt{4\pi }}\phi (u_2)a_{00}(t)+\sum_{q=2}^{\infty }A(u_2,t;q)\right] \\
&=\mathbb{E}[A(u_{1},t)]-\frac{\phi (u_{1})}{\phi (u_{2})}\mathbb{E}%
[A(u_{2},t)]+\sum_{q=2}^{\infty }A(u_1,t;q)-\frac{\phi (u_{1})}{\phi (u_{2})}%
\sum_{q=2}^{\infty }A(u_2,t;q)\,.
\end{align*}
Since the first chaotic components $1/\sqrt{4\pi}\phi (u_{i})a_{00}(t)\sim
I(d_{0})$ of the excursion areas at different levels $u_{i}, \,
i=1,\dots,p_{1}$ have been cancelled by the premultiplication by $\gamma_{1}$%
, it is clear that the order of integration of $\gamma_1
A(u_{1},...,u_{p_{1}};t)$ is determined by the second chaotic components of
the remaining terms, that is
\begin{equation*}
\gamma_1 A(u_{1},...,u_{p_{1}};t)\sim I\left(2d_{\ast}-\frac12\right)\,.
\end{equation*}

Now assume that $d_{0}<2d_{\ell^{\ast }}-\frac{1}{2}$. Take the first row of
$\widetilde{\Gamma }$, that is the $1\times p_1$ vector
\begin{equation*}
\widetilde\gamma_1=\widetilde\gamma_1(u_{1},u_{2}):=\left(%
\begin{array}{ccccc}
1 & -\frac{u_{1}\phi (u_{1})}{u_{2}\phi (u_{2})} & 0 & \cdots & 0%
\end{array}%
\right)\,,
\end{equation*}
then
\begin{align*}
&\widetilde\gamma_1 A(u_{1},...,u_{p_{1}};t)= A(u_{1};t)-\frac{u_{1}\phi
(u_{1})}{u_{2}\phi (u_{2})}A(u_{2};t) \\
&=\mathbb{E}[A(u_{1};t)]-\frac{u_{1}\phi (u_{1})}{u_{2}\phi (u_{2})}\mathbb{E%
}[A(u_{2};t)]+\frac{1}{\sqrt{4\pi }}\phi (u_1)\left(1-\frac{u_1}{u_2}%
\right)a_{00}(t) \\
&\quad +\frac{\phi(u_1)}{6}\left[H_2(u_1)-\frac{u_1}{u_2}H_2(u_2)\right]%
\int_{S^2}H_3(Z(x,t))dx+R_4(u_1,u_2)\,,
\end{align*}
where $R_4(u_1,u_2)$ is just the sum of the chaotic components of order $%
q\ge4$. Since the second chaotic components of the excursion areas at
different levels $u_{i}, \, i=1,\dots,p_{1}$ have been cancelled by the
premultiplication by $\widetilde\gamma_{1}$, it is clear that the order of
integration of $\widetilde\gamma_1 A(u_{1},...,u_{p_{1}};t)$ is determined
by either the first or the third chaotic components of the remaining terms,
that is (see also Section \ref{sec:memory})
\begin{equation*}
\widetilde\gamma_1 A(u_{1},...,u_{p_{1}};t)\sim I(d_{0}\vee3d_\ast-1)\,,
\end{equation*}
and this concludes the proof.
\end{proof}

It is important to stress that fractional cointegrating relationships are
not unique - not only in the sense that they actually live in a $p_{1}-1$
dimensional space spanned by the vectors we listed in Proposition \ref{prop1}%
, but also in a more subtle sense. Indeed, there can exist stronger
cointegrating relationships, meaning linear combinations that have shorter
and shorter memory. More precisely, in the next proposition we show that for
a suitably large vector evaluated at different thresholds, there always
exist a cointegrating relationship yielding short memory residuals. Indeed,
define the $(q_{\ast }+1)\times p_{1}$ matrix, with $p_{1}\geq q_{\ast }+1$,
\begin{equation*}
X_{A(u_{1},...,u_{p_{1}};t)}^{q_{\ast }}=X_{A}:=\left(
\begin{array}{ccc}
\phi (u_{1}) & \cdots & \phi (u_{p_{1}}) \\
\phi (u_{1})H_{1}(u_{1}) & \cdots & \phi (u_{p_{1}})H_{1}(u_{p_{1}}) \\
\vdots & \ddots & \vdots \\
\phi (u_{1})H_{q_{\ast }}(u_{1}) & \cdots & \phi (u_{p_{1}})H_{q_{\ast
}}(u_{p_{1}})%
\end{array}%
\right) \,;
\end{equation*}%
we have the following statement.

\begin{proposition}
\label{prop2} Let
\begin{equation*}
q_{\ast }=\left[ \frac{1}{1-2d_{\ast }}\right] +1\,,
\end{equation*}%
where $[\cdot ]$ denotes the integer part; then any $p_{1}\times 1$ vector $%
\gamma _{2}$ such that $\gamma _{2}^{\prime }X_{\mathcal{A}}=0$ is
cointegrating: in particular,
\begin{equation*}
\gamma _{2}^{\prime }A(u_{1},...,u_{p_{1}};t)\sim I(0).
\end{equation*}
\end{proposition}

\begin{proof}
Recall the chaotic expansion of the area functionals, which is given by
\begin{equation*}
A(u)=\mathbb{E}[A(u)]+\phi (u)J_{0}+H_{1}(u)\phi (u)J_{1}+H_{2}(u)\phi
(u)J_{2}+H_{3}(u)\phi (u)J_{3}+\cdots \,,
\end{equation*}%
where $J_{k}=\frac{1}{(k+1)!}\int_{\mathbb{S}^{2}}H_{k+1}(Z(x,t))\,dx$. Let $%
\gamma _{2}$ be a $p_{1}\times 1$ vector such that $\gamma _{2}^{\prime }X_{%
\mathcal{A}}=0$, then
\begin{equation*}
\gamma _{2}^{\prime }\left(
\begin{array}{c}
\phi (u_{1})H_{k-1}(u_{1}) \\
\phi (u_{2})H_{k-1}(u_{2}) \\
\vdots \\
\phi (u_{p_{1}})H_{k-1}(u_{p_{1}})%
\end{array}%
\right) =0\,,\quad \mbox{for all}\quad k=1,...,q+1\,.
\end{equation*}%
As a consequence, denoting $\gamma _{2}:=(\gamma _{2,1}\quad \cdots \quad
\gamma _{2,p_{1}})^{\prime }$, we have
\begin{align*}
& \gamma _{2}^{\prime }A(u_{1},...,u_{p_{1}};t)=\gamma _{2}^{\prime }\left(
\begin{array}{c}
A(u_{1},t) \\
A(u_{2},t) \\
\vdots \\
A(u_{p_{1}},t)%
\end{array}%
\right) \\
& =\gamma _{2,1}A(u_{1},t)+\cdots +\gamma _{2,p_{1}}A(u_{p_{1}},t)
\end{align*}%
\begin{align*}
& =\gamma _{2}^{\prime }\left(
\begin{array}{c}
\mathbb{E}[A(u_{1})] \\
\mathbb{E}[A(u_{2})] \\
\vdots \\
\mathbb{E}[A(u_{p_{1}})]%
\end{array}%
\right) +J_{0}\,\gamma _{2}^{\prime }\left(
\begin{array}{c}
\phi (u_{1}) \\
\phi (u_{2}) \\
\vdots \\
\phi (u_{p_{1}})%
\end{array}%
\right) +\cdots +J_{q^{\ast }}\,\gamma _{2}^{\prime }\left(
\begin{array}{c}
\phi (u_{1})H_{q^{\ast }}(u_{1}) \\
\phi (u_{2})H_{q^{\ast }}(u_{2}) \\
\vdots \\
\phi (u_{p_{1}})H_{q^{\ast }}(u_{p_{1}})%
\end{array}%
\right) \\
& \quad +\gamma _{2}^{\prime }\left(
\begin{array}{c}
\sum_{q\geq q_{\ast }+1}A(u_{1},t;q) \\
\sum_{q\geq q_{\ast }+1}A(u_{2},t;q) \\
\vdots \\
\sum_{q\geq q_{\ast }+1}A(u_{p_{1}},t;q)%
\end{array}%
\right)
\end{align*}%
\begin{equation*}
=\gamma _{2}^{\prime }\left(
\begin{array}{c}
\mathbb{E}[A(u_{1})] \\
\mathbb{E}[A(u_{2})] \\
\vdots \\
\mathbb{E}[A(u_{p_{1}})]%
\end{array}%
\right) +\gamma _{2}^{\prime }\left(
\begin{array}{c}
\displaystyle\sum_{q\geq q_{\ast }+1}A(u_{1},t;q) \\
\displaystyle\sum_{q\geq q_{\ast }+1}A(u_{2},t;q) \\
\vdots \\
\displaystyle\sum_{q\geq q_{\ast }+1}A(u_{p_{1}},t;q)%
\end{array}%
\right) \sim I(0)\quad \text{by construction}.
\end{equation*}%
Indeed, one has that
\begin{align*}
& \text{Cov}\left( \displaystyle\sum_{q\geq q_{\ast }+1}A(u,0;q),%
\displaystyle\sum_{q\geq q_{\ast }+1}A(u,\tau ;q)\right) \\
& \approx C_{\ell }(\tau )^{q_{\ast }+1}\simeq \left( C_{\ell }(0)(1+\tau
)^{2d_{\ast }-1}\right) ^{q_{\ast }+1}\,,
\end{align*}%
which is integrable since $q_{\ast }=\left[ (1-2d_{\ast })^{-1}\right] +1$.
\end{proof}

\begin{remark}
Of course, there ia variety of intermediate cases which could also be
considered; take for instance $p_{1}=3$ and $q=1$, hence a $q$ which is
smaller than $\left[ (1-2d_{\ast })^{-1}\right] +1$. Then we can still find
a vector $\gamma _{2}$ that cancels out only the first two chaoses for the
three excursion levels we are considering. That is, we solve
\begin{equation*}
(1\quad \gamma _{2,2}\quad \gamma _{2,3})\left(
\begin{array}{c}
\phi (u_{1})H_{k-1}(u_{1}) \\
\phi (u_{2})H_{k-1}(u_{2}) \\
\phi (u_{3})H_{k-1}(u_{3})%
\end{array}%
\right) =0\,,\quad \mbox{for}\quad k=1,2\,,
\end{equation*}%
and obtain
\begin{equation*}
\gamma _{2}=\left( 1,\,\frac{\phi (u_{1})}{\phi (u_{2})}\left( \frac{%
H_{1}(u_{1})-H_{1}(u_{3})}{H_{1}(u_{3})-H_{1}(u_{2})}\right) ,\,-\frac{\phi
(u_{1})}{\phi (u_{3})}\left( \frac{H_{1}(u_{1})-H_{1}(u_{2})}{%
H_{1}(u_{3})-H_{1}(u_{2})}\right) \right) .
\end{equation*}%
As a consequence,
\begin{align*}
& \gamma _{2}A(u_{1},u_{2},u_{3};t) \\
& =\phi (u_{1})J_{0}+H_{1}(u_{1})\phi (u_{1})J_{1}+\sum_{q\geq
3}A(u_{1},t)[q] \\
& +\frac{\phi (u_{1})}{\phi (u_{2})}\left( \frac{H_{1}(u_{1})-H_{1}(u_{3})}{%
H_{1}(u_{3})-H_{1}(u_{2})}\right) \left[ \phi (u_{2})J_{0}+H_{1}(u_{2})\phi
(u_{2})J_{1}+\sum_{q\geq 3}A(u_{2},t)[q]\right] \\
& -\frac{\phi (u_{1})}{\phi (u_{3})}\left( \frac{H_{1}(u_{1})-H_{1}(u_{2})}{%
H_{1}(u_{3})-H_{1}(u_{2})}\right) \left[ \phi (u_{3})J_{0}+H_{1}(u_{3})\phi
(u_{3})J_{1}+\sum_{q\geq 3}A(u_{3},t)[q]\right] \\
& =\left[ 1+\left( \frac{H_{1}(u_{1})-H_{1}(u_{3})}{H_{1}(u_{3})-H_{1}(u_{2})%
}\right) -\left( \frac{H_{1}(u_{1})-H_{1}(u_{2})}{H_{1}(u_{3})-H_{1}(u_{2})}%
\right) \right] \phi (u_{1})J_{0} \\
& +\left[ H_{1}(u_{1})+H_{1}(u_{2})\left( \frac{H_{1}(u_{1})-H_{1}(u_{3})}{%
H_{1}(u_{3})-H_{1}(u_{2})}\right) -H_{1}(u_{3})\left( \frac{%
H_{1}(u_{1})-H_{1}(u_{2})}{H_{1}(u_{3})-H_{1}(u_{2})}\right) \right] \phi
(u_{1})J_{1} \\
& +\sum_{q\geq 3}\left\{ A(u_{1},t)[q]+\gamma
_{22}(u_{1},u_{2},u_{3})A(u_{2},t)[q]+\gamma
_{23}(u_{1},u_{2},u_{3})A(u_{3},t)[q]\right\} \\
& =\sum_{q\geq 3}\left\{ A(u_{1},t)[q]+\gamma
_{22}(u_{1},u_{2},u_{3})A(u_{2},t)[q]+\gamma
_{23}(u_{1},u_{2},u_{3})A(u_{3},t)[q]\right\} \,,
\end{align*}%
which means that $\gamma _{2}A(u_{1},u_{2},u_{3};t)\sim I(3d_{\ast }-1)$. In
particular, in the case $d_{\ast }$ is smaller than $1/3$ ($\implies 1=q>%
\left[ (1-2d_{\ast })^{-1}\right] $), we have $\gamma
_{2}A(u_{1},u_{2},u_{3};t)\sim I(0)$.
\end{remark}

Following the same steps as in the proof of Propositions \ref{prop1} and \ref%
{prop2}, we obtain also the following results (whose proofs follow simolar
steps and hence are omitted).

\begin{proposition}
For $v_{1},...,v_{p_{2}}\in \mathbb{R}\backslash 0,$ and $d_{0}>2d_{\ell
^{\ast }}-\frac{1}{2},$ the vector process $\mathcal{L(}v_{1},...,v_{p_{2}};t%
\mathcal{)}$ is fractionally cointegrated of rank $p_{2}-1$ and order:
\begin{equation*}
CI\left(d_0,d_0-(2d_{\ast}-{1}/{2})\right).
\end{equation*}

For $v_{1},...,v_{p_2}\in \mathbb{R}\backslash 0,$ and $d_{0}<2d_{\ast}-%
\frac{1}{2},$ the vector process $\mathcal{L(}v_{1},...,v_{p_{2}};t\mathcal{)%
}$ is fractionally cointegrated of rank $p_{2}-2$ and order:
\begin{equation*}
CI\left(2d_{\ast}-{1}/{2},(2d_{\ast}-{1}/{2})-d_0\right).
\end{equation*}

The cointegrating space in the first case is spanned by the vectors
orthogonal to $\left( v_{1}\phi (v_{1}),v_{2}\phi (v_{2}),...,v_{p_{2}}\phi
(v_{p_{2}})\right) ;$ in the second case, it is spanned by the vectors
orthogonal to the matrix%
\begin{equation*}
\left(
\begin{array}{ccccc}
\phi (v_{1}) & \phi (v_{2}) & ... & ... & \phi (v_{p_{2}}) \\
(v_{1}^{2}-1)\phi (v_{1}) & (v_{2}^{2}-1)\phi (v_{2}) & ... & ... &
(v_{p_{2}}^{2}-1)\phi (v_{p_{2}})%
\end{array}%
\right) \,.
\end{equation*}%
For $d_{0}=2d_{\ast}-\frac{1}{2},$ the vector process $\mathcal{L(}%
v_{1},...,v_{p_{2}};t\mathcal{)}$ is fractionally cointegrated of rank $%
p_{2}-3,$ and the cointegrating space is spanned by the vectors orthogonal
to the matrix%
\begin{equation*}
\left(
\begin{array}{ccccc}
\phi (v_{1}) & \phi (v_{2}) & ... & ... & \phi (v_{p_{2}}) \\
v_{1}\phi (v_{1}) & v_{2}\phi (v_{2}) & ... & ... & v_{p_{2}}\phi (v_{p_{2}})
\\
(v_{1}^{2}-1)\phi (v_{1}) & (v_{2}^{2}-1)\phi (v_{2}) & ... & ... &
(v_{p_{2}}^{2}-1)\phi (v_{p_{2}})%
\end{array}%
\right)
\end{equation*}
\end{proposition}

\begin{proposition}
The joint vector process $(A(u_{1},...,u_{p_{1}};t),\mathcal{L(}%
v_{1},...,v_{p_{2}};t\mathcal{))}$ is fractionally cointegrated of order $%
p_{1}+p_{2}-1$, for $d_{0}>2d_{\ast}-\frac{1}{2};$ of order $p_{1}+p_{2}-2,$
for $d_{0}<2d_{\ast}-\frac{1}{2};$ and of order $p_{1}+p_{2}-3$ for $%
d_{0}=2d_{\ast}-\frac{1}{2}.$ The cointegrating space is spanned by the
vectors which are orthogonal to the matrices (respectively)%
\begin{equation*}
\left(
\begin{array}{cccccc}
\phi (u_{1}) & ... & \phi (u_{p_{1}}) & v_{1}\phi (v_{1}) & ... &
v_{p_{2}}\phi (v_{p_{2}})%
\end{array}%
\right) \,,
\end{equation*}%
\begin{equation*}
\left(
\begin{array}{cccccc}
(u_{1}^{2}-1)\phi (u_{1}) & ... & (u_{p_{1}}^{2}-1)\phi (u_{p_{1}}) & \phi
(v_{1}) & ... & \phi (v_{p_{2}}) \\
(u_{1}^{2}-1)\phi (u_{1}) & ... & (u_{1}^{2}-1)\phi (u_{1}) &
(v_{1}^{2}-1)\phi (v_{1}) & ... & (v_{p_{2}}^{2}-1)\phi (v_{p_{2}})%
\end{array}%
\right) \,,
\end{equation*}%
\begin{equation*}
\left(
\begin{array}{cccccc}
\phi (u_{1}) & ... & \phi (u_{p_{1}}) & v_{1}\phi (v_{1}) & ... &
v_{p_{2}}\phi (v_{p_{2}}) \\
(u_{1}^{2}-1)\phi (u_{1}) & ... & (u_{p_{1}}^{2}-1)\phi (u_{p_{1}}) & \phi
(v_{1}) & ... & \phi (v_{p_{2}}) \\
(u_{1}^{2}-1)\phi (u_{1}) & ... & (u_{1}^{2}-1)\phi (u_{1}) &
(v_{1}^{2}-1)\phi (v_{1}) & ... & (v_{p_{2}}^{2}-1)\phi (v_{p_{2}})%
\end{array}%
\right) \,.
\end{equation*}
\end{proposition}

\subsection{The averaged periodogram and a conjecture}

In this section, we present the averaged periodogram estimator introduced by
Robinson in 1994, see \cite{Robinson1994}, and then we present a conjecture
about the extension of those results to Hermite polynomial transforms of
stationary long memory Gaussian processes.\newline

Consider a centered process $\{X_{t}\}_{t\in \mathbb{Z}}$ whose covariance
function is given by $\mathbb{E}[X_{t}X_{s}]=\rho (t-s)$ and spectral
density $f_{X}(\lambda )=\frac{1}{2\pi }\sum_{\tau =-\infty }^{\infty }\rho
(\tau )\cos (\tau \lambda ).$ Let us recall the spectral distribution
function
\begin{equation*}
F_{X}(\lambda )=\int_{0}^{\lambda }f(\theta )d\theta
\end{equation*}%
and the averaged periodogram
\begin{equation*}
\widehat{F}_{X}(\lambda )=\frac{2\pi }{T}\sum_{j=1}^{[T\lambda /2\pi
]}I_{X}(\lambda _{j})
\end{equation*}%
where $[\cdot ]$ denotes the integer part, $\lambda _{j}:=2\pi j/T$, and we
have introduced also the sample periodogram, given by
\begin{equation*}
I_{X}(\lambda )=|w_{X}(\lambda )|^{2}\qquad w_{X}(\lambda )=\frac{1}{\sqrt{%
2\pi T}}\sum_{t=1}^{T}X_{t}\exp (-i\lambda t)\text{ .}
\end{equation*}%
Now we need to introduce a few regularity conditions, which follow \cite%
{Robinson1994}.

\begin{itemize}
\item \textbf{Condition A.} (\emph{Long Memory Behaviour}) For some $d\in
\left( 0,\frac{1}{2}\right) $,
\begin{equation*}
f_{X}(\lambda )\simeq L\left( \frac{1}{\lambda }\right) \lambda ^{-2d},
\end{equation*}%
where $L(\lambda )$ is a slowly varying function at infinity, see \cite{BGT}.

\item \textbf{Condition B.} (\emph{Bandwidth Condition}) $\frac{1}{m}+\frac{m%
}{T}\rightarrow 0$ as $T\rightarrow \infty $.

\item \textbf{Condition C.} \ (\emph{Linearity}) The process $%
\{X_{t}\}_{t\in \mathbb{Z}}$ has the following Wold representation
\begin{equation*}
X_{t}=\sum_{j=0}^{\infty }\alpha _{j}e_{t-j}\,,\qquad \sum_{j=0}^{\infty
}\alpha _{j}^{2}<\infty
\end{equation*}%
where

\begin{itemize}
\item[$(i)$] $\mathbb{E}[e_{t}e_{s}]=0\,, \quad t>s$;

\item[$(ii)$] $\mathbb{E}[e_{u}e_{v}e_{t}e_{s}]=%
\begin{cases}
\sigma^{4} & \quad t=s>u=v \\
0 & \quad \text{otherwise}%
\end{cases}%
;$

\item[$(iii)$] There exists a nonnegative random variable $e$ such that, for
all $\varepsilon>0$ and some $K<1$,
\begin{equation*}
\mathbb{E }[e^{2}] <\infty \qquad \mathbb{P}(|e_{t}|>\varepsilon)\le K\,
\mathbb{P}(|e|>\varepsilon);
\end{equation*}

\item[$(iv)$] $\displaystyle\frac{1}{T}\sum_{t=1}^{T}\mathbb{E}\left[
e_{t}^{2}|e_{s}^{2}\,,\,\,s<t\right] \rightarrow _{p}\sigma ^{2}\quad \text{%
as}\quad T\rightarrow \infty ;$
\end{itemize}
\end{itemize}

We are in the position to recall the following result, which was given in
\cite[Theorem 1]{Robinson1994}:\newline

\begin{theorem}
(See \cite{Robinson1994}) Let Conditions A, B and C hold. Then
\begin{equation*}
\frac{\widehat{F}_{X}(\lambda _{m})}{F_{X}(\lambda _{m})}\rightarrow
_{p}1\qquad \text{as}\quad T\rightarrow \infty .
\end{equation*}
\end{theorem}

Heuristically, the result can be illustrated as follows: it is classical in
time series analysis that a form of locally averaged periodogram converges
under regularity conditions to the spectral density. In the case of long
range dependence, these results can not even be formulated, as the spectral
density diverges at the origin; however, it can be still shown that the sum
of the periodogram ordinates converge to the integral of the spectral
density, in a neigbourhood of the origin.

\bigskip

Now consider our Gaussian isotropic and stationary sphere-cross-time random
field $\{Z(x,t),\ (x,t)\in \mathbb{S}^{2}\times \mathbb{Z}\}$ and assume
that
\begin{equation*}
\widetilde{d}_{\ast }\in \left( \frac{q-1}{2q},\frac{1}{2}\right) \qquad
\text{with}\quad \widetilde{d}_{\ast }=d_{\ast }\vee d_{0}\,.
\end{equation*}%
Denote by
\begin{equation*}
\mathcal{X}_{t}^{q}:=\int_{\mathbb{S}^{2}}H_{q}\left( Z(x,t)\right) dx
\end{equation*}%
and consider the corresponding averaged periodogram
\begin{equation*}
\widehat{F}_{q}(\lambda )=\frac{2\pi }{T}\sum_{j=1}^{[T\lambda /2\pi
]}I_{q}(\lambda _{j})
\end{equation*}%
where
\begin{equation*}
I_{q}(\lambda ):=|w_{q}(\lambda )|^{2}\qquad w_{q}(\lambda ):=\frac{1}{\sqrt{%
2\pi T}}\sum_{t=1}^{T}\mathcal{X}_{t}^{q}\exp (-i\lambda t)\text{ .}
\end{equation*}%
Using the diagram formula, (see \cite[Proposition 4.15]{MaPeCUP} and also
Section \ref{sec:memory}), we have immediately
\begin{align*}
\rho _{q}(\tau )& :=\mathbb{E}\left[ \mathcal{X}_{0}^{q}\mathcal{X}_{\tau
}^{q}\right] =q!\,\int_{\mathbb{S}^{2}}\int_{\mathbb{S}^{2}}\Gamma (\langle
x,y\rangle ,\tau )^{q}\,dx\,dy \\
& \simeq L_{q}(1+\tau )^{q(2d_{\ast }-1)}=L_{q}(1+\tau )^{2d_{q}-1}\text{ }
\\[5pt]
\text{with}\quad d_{q}& :=q\widetilde{d}_{\ast }-\frac{q-1}{2} \\
\text{and}\quad L_{q}&=q! \int_{\mathbb{S}^{2}\times \mathbb{S}^{2}}\Bigg(%
\sum_{\ell:d_{\ell}=\widetilde{d}_{\ast}}
(2\ell+1)C_{\ell}(0)P_{\ell}(\langle x,y\rangle)\Bigg)^{q}dx\,dy\text{.}
\end{align*}%
Let us also recall the following Tauberian-type results (see \cite[Theorem
4.1.10]{BGT}). As usual, we say that a function is slowly-varying if and
only if
\begin{equation*}
\lim_{\ell \rightarrow \infty }\frac{L(c\ell )}{L(\ell )}=1\text{ , for all }%
c>0\text{ .}
\end{equation*}

\begin{theorem}
\label{thm:tauber} (See \cite{BGT}) Assume $L(\cdot )$ is slowly varying and
$\rho (\tau )\geq 0$ for $\tau $ large enough. Then we have that, if
\begin{equation*}
\lim_{\lambda \rightarrow 0}\frac{f(\lambda )}{\pi ^{-1}\Gamma
(2d)L(1/\lambda )\lambda ^{-2d}\sin \left( \frac{\pi (1-2d)}{2}\right) }=1
\end{equation*}%
then%
\begin{equation*}
\lim_{n\rightarrow \infty }\frac{2d}{n^{2d}L(n)}\sum_{\tau =0}^{n}\rho (\tau
)=1\text{ .}
\end{equation*}%
Moreover, assuming that
\begin{equation*}
\lim_{\eta \rightarrow 1}\liminf_{\tau \rightarrow \infty }\min_{\tau \leq
\tau ^{\prime }\leq \eta \tau }\frac{\rho (\tau ^{\prime })-\rho (\tau )}{%
\tau ^{2d-1}L(\tau )}\geq 0
\end{equation*}%
then%
\begin{equation*}
\lim_{\tau \rightarrow \infty }\frac{\rho (\tau )}{L(\tau )\tau ^{2d-1}}=1%
\text{ .}
\end{equation*}
\end{theorem}

\begin{remark}
Note that the condition on the autocovariances is trivially satisfied if
they are monotonically decreasing. Corollary 4.10.2 in the same reference
implies that if $\rho (\tau )$ is monotonic then the converse statement is
also true.
\end{remark}

Using the previous Theorem \ref{thm:tauber}, we have that
\begin{equation*}
f_{q}(\lambda )=\frac{1}{2\pi }\sum_{\tau =-\infty }^{\infty }\rho _{q}(\tau
)e^{-i\tau \lambda }\simeq \pi ^{-1}\Gamma (2d_{q})L_{q}\lambda
^{-2d_{q}}\sin ({\pi (1-2d_{q})}/{2}) \lambda ^{-2d_{q}}
\end{equation*}%
and hence we define
\begin{equation*}
{F}_{q}(\lambda )= \pi ^{-1}\Gamma (2d_{q})L_{q}\lambda ^{-2d_{q}}\sin ({\pi
(1-2d_{q})}/{2}) \int_{0}^{\lambda _{m}}\lambda ^{-2d_{q}}d\lambda \text{ .}
\end{equation*}%
Note that, given the assumption $q>\left( 1-2\widetilde{d}_{\ast }\right)
^{-1}$, the process $\{\mathcal{X}_{t}^{q}\}_{t}$ is still long memory,
since this implies that $0<d_{q}<\frac{1}{2}\,$. However, it should be noted
that the \textquotedblleft Hermite\textquotedblright\ processes that we are
considering in this paper satisfy Condition C in \cite{Robinson1994} only
for $q=1$; otherwise, they are characterized by a nonlinear structure which
does not necessarily fulfill such a condition. Let us define%
\begin{equation*}
\widehat{F}_{q}(\lambda _{m})=\frac{2\pi }{T}\sum_{j=1}^{[T\lambda _{m}/2\pi
]}I_{q}(\lambda _{j})=\frac{1}{T^{2}}\sum_{j=1}^{m}\sum_{t,s=1}^{T}\mathcal{X%
}_{t}^{q}\mathcal{X}_{s}^{q}e^{-i\lambda _{j}(t-s)}
\end{equation*}%
We believe that a result analogous to Theorem 1 in \cite{Robinson1994} holds
in this generalized setting as well; in particular, we propose the following

\begin{conjecture}
\label{conj}Under Condition B and the previous assumptions and notations on $%
Z(x,t)$, we have that
\begin{equation*}
\frac{\widehat{F}_{q}(\lambda _{m})}{F_{q}(\lambda _{m})}\rightarrow _{p}1\,.
\end{equation*}
\end{conjecture}

We believe that Conjecture \ref{conj} can be proved along lines similar to
those in \cite{Robinson1994}. More precisely, the leading term in the second
moment of the previous ratio can be decomposed into two summands; the first
one corresponds to the squared expected value, which we expect to be
dominating, while the variance term can presumably be shown to be $o(1)$ by
a careful analysis of cumulants of the Fourier transforms, by means of the
diagram formula, see \cite[Proposition 4.15]{MaPeCUP}. A full investigation
of this argument however seems demanding and it is beyond the purpose of
this paper; as we believe these results would have independent interest, we
leave them for further research.

In the next section we explore a statistical estimator for a crucial
parameter in sphere-cross-time random fields, under the assumption that the
previous conjecture holds.

\subsection{Statistical applications: estimation of $\protect\sigma _{1}$}

In this section, we show how the existence of geometric fractional
cointegrating relationships can be exploited to produce very simple and
practical estimators for the most important parameter when studying the
excursion sets of Gaussian random fields. In particular, it is well-known
that, under isotropy, the expected value of the boundary length for
isotropic random fields, and more generally the expected value for all
Lipschitz-Killing curvatures (thus including also the Euler-Poincar\'{e}
characteristic, which we do not study in this note) can be expressed in
terms of a single free parameter: the so-called first spectral moment of the
field. The latter corresponds to the variance of the first derivative of the
field itself; \ for our purposes it is sufficient to recall that the
Kac-Rice formula gives (see e.g., \cite{AdlerTaylor}, \cite{AHL2024} )
\begin{equation*}
E[\mathcal{L(}u;t\mathcal{)]=\sigma }_{1}2\pi e^{-u^{2}/2},\qquad \sigma
_{1}^{2}=\sum_{\ell =0}^{\infty }\frac{2\ell +1}{4\pi }C_{\ell }\frac{\ell
(\ell +1)}{2}\,.
\end{equation*}%
A natural way to estimate $\sigma _{1}$ would then be to consider the
statistic
\begin{equation*}
\widehat{\sigma }_{1}:=\frac{1}{2\pi e^{-u^{2}/2}}\frac{1}{T}\sum_{t=1}^{T}%
\mathcal{L(}u;t\mathcal{)}
\end{equation*}%
which is clearly unbiased, with asymptotic variance of order%
\begin{equation*}
\text{Var}(\widehat{\sigma }_{1})=\frac{1}{4\pi ^{2}e^{-u^{2}}}\frac{1}{T}%
\sum_{\tau =1-T}^{T-1}\left( 1-\frac{|\tau|}{T}\right) \text{Cov}(\mathcal{L(%
}u;0\mathcal{)},\mathcal{L(}u;\tau \mathcal{))}\approx
T^{2d_{0}-1}+T^{4d_{\ast }-2}.
\end{equation*}%
However, this estimator need not be very efficient, in fact the previous
rates of convergence can become arbitrarily slow when $d_{0}$ or $d_{\ast }$
are close to $\frac{1}{2}$. Our idea is to introduce an alternative
estimator which, under Conjecture \ref{conj}, can lead to faster consistency
rates.

\subsubsection{A consistent estimator under Conjecture \protect\ref{conj}}

Now we must make the following distinction:

\begin{itemize}
\item \textbf{case a}. $d_0>2d_{\ast}-\frac12$;

\item \textbf{case b}. $d_0<2d_{\ast}-\frac12$ and there exists a unique $%
\ell=\ell^\ast$ such that $d_{\ell}=d_{\ast}$.
\end{itemize}

Note that in \textbf{case a} the residuals are dominated by the second
chaos, whereas in \textbf{case b} by the first. Let us indeed recall the
chaotic expansions \eqref{eq:chaotic-exp-A} and \eqref{eq:chaotic-exp-L}. In
\textbf{case a} we rewrite them as follows:
\begin{align*}
A(u;t)-\mathbb{E}[A(u;t)]& =\frac{1}{\sqrt{4\pi }}\phi
(u)a_{00}(t)+R_{A}(u,t)\,, \\
\mathcal{L}(u;t)-\mathbb{E}[\mathcal{L}(u;t)]& =\sigma _{1}\sqrt{2}\pi u\phi
(u)a_{00}(t)+R_{\mathcal{L}}(u;t)\,,
\end{align*}%
where $R_{A}(u,t):=\sum_{q=2}^{\infty }{A}(u,t;q)$ and $R_{\mathcal{L}%
}(u;t)=\sum_{q=2}^{\infty }\mathcal{L}(u,t)[q]$.

Then, we obtain
\begin{align*}
\mathcal{L(}u;t\mathcal{)-}\mathbb{E}[\mathcal{L(}u;t\mathcal{)]} &=\beta
(u,\sigma _{1})\left\{ A(u,t)-\mathbb{E}[A(u,t)]\right\} +R_{\mathcal{L}%
}(u;t)-\beta (u,\sigma _{1}) \,R_{A}(u,t)
\end{align*}%
where%
\begin{align*}
\beta=\beta (u,\sigma _{1}) &:=\sigma _{1}\sqrt{8}\pi ^{3/2}u\,,
\end{align*}

For brevity sake, now denote $X^u_t:=A(u;t)-\mathbb{E}[A(u;t)]$, $Y^u_t:=%
\mathcal{L(}u;t\mathcal{)-}\mathbb{E}[\mathcal{L(}u;t\mathcal{)]}$ and $%
\varepsilon_t^u :=R_{\mathcal{L}}(u;t)-\beta \,R_{A}(u,t)$. We have hence
achieved the following regression form
\begin{equation*}
Y^u_t=\beta X^u_t + \varepsilon_t^u\,,
\end{equation*}
with two main problems:

\begin{itemize}
\item both the dependent variable and the regressor depend on an unknown
expected value, which is itself function of the unknown parameter $\sigma
_{1}$;

\item the regressor and dependent variable are correlated.
\end{itemize}

Both problems are solved resorting to Narrow-Band Least Squares. Consider
indeed the Discrete Fourier Transforms%
\begin{eqnarray*}
w_{Y_t^u}(\lambda _{j}) &=&\frac{1}{\sqrt{2\pi T}}\sum_{t=1}^{T}Y_t^u\exp
(-i\lambda _{j}t)\,, \\
w_{X_t^u}(\lambda _{j}) &=&\frac{1}{\sqrt{2\pi T}}\sum_{t=1}^{T}X_t^u\exp
(-i\lambda _{j}t)\,,
\end{eqnarray*}%
and introduce the narrow-band least square estimator
\begin{equation*}
\widehat{\beta}:=\frac{\sum_{j=1}^{m}w_{X_t^u}(\lambda _{j})\overline{w}%
_{Y_t^u}(\lambda _{j})}{\sum_{j=1}^{m}w_{X_t^u}(\lambda _{j})\overline{w}%
_{X_t^u}(\lambda _{j})}\,.
\end{equation*}
which implies that a narrow-band least square estimator for $\widehat{\sigma}%
_{1}$ is given by
\begin{equation*}
\widehat{\sigma}_{1}(u):=\frac{1}{\sqrt{8}\pi^{3/2} u}\widehat{\beta}\,.
\end{equation*}

Assuming that Conjecture \ref{conj} holds, we have the following

\begin{proposition}
\label{prop:consistency-case-a} In $\text{\textbf{case a}}$, that is $%
d_0>2d_{\ast}-\frac12$, we have, as $\frac1m+\frac mT\rightarrow 0$, that
the estimator $\widehat{\sigma}_{1}$ is consistent and
\begin{equation*}
|\widehat{\sigma}_{1}-\sigma _{1}|=O_{p}\left(\left(\frac
mT\right)^{d_{0}-(2d_\ast-1/2)}\right)\,.
\end{equation*}
\end{proposition}

\begin{proof}
We want to prove that $\widehat{\sigma }_{1}$ is consistent, that is $%
\widehat{\sigma }_{1}\rightarrow \sigma _{1}$ in probability, as $%
m,T\rightarrow \infty $.%
We have
\begin{align*}
\widehat{\beta }& =\frac{\sum_{j=1}^{m}w_{X_{t}^{u}}(\lambda _{j})(\beta
\overline{w}_{X_{t}^{u}}(\lambda _{j})+\overline{w}_{\varepsilon
_{t}^{u}}(\lambda _{j}))}{\sum_{j=1}^{m}|w_{X_{t}^{u}}(\lambda _{j})|^{2}} \\
& =\beta +\frac{\sum_{j=1}^{m}w_{X_{t}^{u}}(\lambda _{j})\overline{w}%
_{\varepsilon _{t}^{u}}(\lambda _{j})}{\sum_{j=1}^{m}|w_{X_{t}^{u}}|^{2}}
\end{align*}%
and hence, following \cite[\S 5.3]{Robinson1994}, we have
\begin{align*}
|\widehat{\beta }-\beta |& =\left\vert \frac{\sum_{j=1}^{m}w_{X_{t}^{u}}(%
\lambda _{j})\overline{w}_{\varepsilon _{t}^{u}}(\lambda _{j})}{%
\sum_{j=1}^{m}|w_{X_{t}^{u}}|^{2}}\right\vert =\left\vert \frac{%
\sum_{j=1}^{m}I_{X\varepsilon }(\lambda _{j})}{\sum_{j=1}^{m}I_{X}(\lambda
_{j})}\right\vert \\
& =\left\vert \frac{\widehat{F}_{X\varepsilon }(\lambda _{m})}{\widehat{F}%
_{X}(\lambda _{m})}\right\vert \leq \left\{ \frac{\widehat{F}_{\varepsilon
}(\lambda _{m})}{\widehat{F}_{X}(\lambda _{m})}\right\} ^{1/2}\,,
\end{align*}%
where the last inequality follows from an application of Cauchy-Schwarz
inequality, that is
\begin{align*}
\left\vert \frac{\widehat{F}_{X\varepsilon }(\lambda _{m})}{\widehat{F}%
_{X}(\lambda _{m})}\right\vert & =\sum_{j=1}^{m}\frac{w_{X}(\lambda _{j})}{%
\sqrt{\widehat{F}_{X}(\lambda _{m})}}\frac{\overline{w}_{\varepsilon
}(\lambda _{j})}{\sqrt{\widehat{F}_{X}(\lambda _{m})}} \\
& \leq \left( \sum_{j=1}^{m}\frac{w_{X}(\lambda _{j})\overline{w}%
_{X}(\lambda _{j})}{\widehat{F}_{X}(\lambda _{m})}\right) ^{1/2}\left(
\sum_{j=1}^{m}\frac{w_{\varepsilon }(\lambda _{j})\overline{w}_{\varepsilon
}(\lambda _{j})}{\widehat{F}_{X}(\lambda _{m})}\right) ^{1/2} \\
& =\left( \frac{\widehat{F}_{X}(\lambda _{m})}{\widehat{F}_{X}(\lambda _{m})}%
\right) ^{1/2}\left( \frac{\widehat{F}_{\varepsilon }(\lambda _{m})}{%
\widehat{F}_{X}(\lambda _{m})}\right) ^{1/2}=\left\{ \frac{\widehat{F}%
_{\varepsilon }(\lambda _{m})}{\widehat{F}_{X}(\lambda _{m})}\right\}
^{1/2}\,.
\end{align*}%
As a consequence, starting from the following inequality
\begin{equation*}
|\widehat{\beta }-\beta |\leq \left\{ \frac{\widehat{F}_{\varepsilon
}(\lambda _{m})}{\widehat{F}_{X}(\lambda _{m})}\right\} ^{1/2}=\left\{ \frac{%
\widehat{F}_{\varepsilon }(\lambda _{m})}{{F}_{\varepsilon }(\lambda _{m})}%
\frac{{F}_{X}(\lambda _{m})}{\widehat{F}_{X}(\lambda _{m})}\right\}
^{1/2}\left\{ \frac{F_{\varepsilon }(\lambda _{m})}{F_{X}(\lambda _{m})}%
\right\} ^{1/2}
\end{equation*}%
and using the fact that (Conjecture \ref{conj})
\begin{equation*}
\frac{\widehat{F}_{\varepsilon }(\lambda _{m})}{{F}_{\varepsilon }(\lambda
_{m})}\frac{{F}_{X}(\lambda _{m})}{\widehat{F}_{X}(\lambda _{m})}\rightarrow
_{p}1
\end{equation*}%
and
\begin{equation*}
\left\{ \frac{F_{\varepsilon }(\lambda _{m})}{F_{X}(\lambda _{m})}\right\}
^{1/2}=\left\{ \frac{\displaystyle\int_{0}^{\lambda _{m}}\lambda
^{-2d_{\varepsilon }}d\lambda }{\displaystyle\int_{0}^{\lambda _{m}}\lambda
^{-2d_{X^{u}}}d\lambda }\right\} ^{1/2}=\lambda
_{m}^{d_{X^{u}}-d_{\varepsilon }}=\left( \frac{2\pi \,m}{T}\right)
^{d_{X^{u}}-d_{\varepsilon }}
\end{equation*}%
where, by construction,
\begin{equation*}
\varepsilon _{t}\sim I\left( 2d_{\ast }-\frac{1}{2}\right) \qquad
X_{t}^{u}\sim I(d_{0})
\end{equation*}%
and hence $d_{\varepsilon }=2d_{\ast }-\frac{1}{2}$ while $d_{X^{u}}=d_{0}$,
we have
\begin{equation*}
|\widehat{\beta }-\beta |=O_{p}\left( \left( \frac{2\pi \,m}{T}\right)
^{d_{0}-(2d_{\ast }-1/2)}\right)
\end{equation*}%
which concludes the proof.
\end{proof}

\bigskip

In \textbf{case b}, we fix the level
\begin{equation*}
u=u^\ast:=\sqrt{2-\frac{\lambda_{\ell^\ast}/2}{\sigma_1^2}} \implies
(u^{2}-1)+\frac{\lambda _{\ell }/2}{\sigma _{1}^{2}}=1\,,
\end{equation*}
and we rewrite the two chaotic expansions \eqref{eq:chaotic-exp-A} and %
\eqref{eq:chaotic-exp-L} as follows:
\begin{align*}
&A(u^\ast;t)-E[A(u^\ast;t)]=u^\ast \phi (u^\ast)(2\ell^\ast+1)\left\{
\widehat C_{\ell^\ast }(t)-C_{\ell^\ast}\right\} + R_A(u^\ast,t)^{\prime } \\
&\mathcal{L}(u^\ast;t)-E[\mathcal{L}(u^\ast;t)]=\frac{\sigma _{1}}{2}\sqrt{%
\frac{\pi }{2}}\phi (u^\ast)(2\ell^\ast+1)\left\{ \widehat C_{\ell^\ast
}(t)-C_{\ell^\ast}\right\}+ R_{\mathcal{L}}(u^\ast,t)^{\prime }
\end{align*}
which implies
\begin{equation*}
\mathcal{L}(u^\ast;t)-E[\mathcal{L}(u^\ast;t)]=\frac{\sigma _{1}}{2}\sqrt{%
\frac{\pi }{2}}\frac{1}{u^\ast} \{A(u^\ast;t)-E[A(u^\ast;t)]\}+ R_{\mathcal{L%
}}(u^\ast,t)^{\prime }-\frac{1}{u^\ast}R_A(u^\ast,t)^{\prime }
\end{equation*}
We have hence achieved the following regression form
\begin{equation*}
Y^{u^\ast}_t=\alpha X^{u^\ast}_t + \eta_t^{u^\ast}\,,
\end{equation*}
where $\alpha=\alpha(\sigma_1)=\frac{\sigma _{1}}{2}\sqrt{\frac{\pi }{2}}%
\frac{1}{u^\ast}$ and $\eta_t^{u^\ast}=R_{\mathcal{L}}(u^\ast,t)^{\prime
}-\alpha R_A(u^\ast,t)^{\prime }$. Following what has been done for \textbf{%
case a}, we introduce the narrow-band least square estimator
\begin{equation*}
\widehat{\alpha}:=\frac{\sum_{j=1}^{m}w_{X_t^{u^\ast}}(\lambda _{j})%
\overline{w}_{Y_t^{u^\ast}}(\lambda _{j})}{\sum_{j=1}^{m}w_{X_t^{u^\ast}}(%
\lambda _{j})\overline{w}_{X_t^{u^\ast}}(\lambda _{j})}\,.
\end{equation*}
which implies that a narrow-band least square estimator for $\widehat{\sigma}%
_{1}$ is given by
\begin{equation*}
\widehat{\sigma}_{1}({u^\ast}):=2{u^\ast} \sqrt{\frac2\pi}\,\widehat{\alpha}%
\,.
\end{equation*}

\begin{proposition}
In $\text{\textbf{case b}}$, that is $d_0<2d_{\ast}-\frac12$ and there
exists a unique $\ell=\ell^\ast$ such that $d_{\ell}=d_{\ast}$, we have, as $%
T,m\rightarrow \infty $, that the estimator $\widehat{\sigma}_{1}$ is
consistent and
\begin{equation*}
\widehat{|\sigma}_{1}-\sigma _{1}|=O_{p}\left(\left(\frac
mT\right)^{(2d_{\ast}-1/2)-[d_{0}\vee (3d_{\ast}-1) \vee
(2d_{\ast\ast}-1/2)]}\right)\,.
\end{equation*}
\end{proposition}

\begin{proof}
Here, by construction, the residuals of the regression $\varepsilon_t^u$
contains the first and third chaoses of $A(u^\ast;t)$ and $\mathcal{L}%
(u^\ast;t)$, and the following term
\begin{equation*}
u^\ast \phi(u^\ast) \sum_{\ell\ne\ell^\ast}(2\ell+1)\left(\alpha\left[%
(u^\ast)^2-1+\frac{\lambda_\ell/2}{\sigma_1^2}\right]-1\right)\left\{%
\widehat{C}_\ell(t)-C_\ell\right\}
\end{equation*}
which are the multipoles of the second chaoses of the area and length
functionals that have shorter memories, and hence its order of integration
is $d_{0}\vee (3d_{\ast}-1) \vee (2d_{\ast\ast}-1/2)$ which means
\begin{equation*}
\varepsilon_{t} \sim I\left(d_{0}\vee (3d_{\ast}-1) \vee
(2d_{\ast\ast}-1/2)\right)
\end{equation*}
As a consequence
\begin{align*}
|\widehat{\alpha}-\alpha| &=O_{p}\left(\left(\frac
mT\right)^{(2d_{\ast}-1/2)-[d_{0}\vee (3d_{\ast}-1) \vee
(2d_{\ast\ast}-1/2)]}\right)\,,
\end{align*}
which concludes the proof.
\end{proof}

\begin{remark}[Estimation of the leading multipoles]
One example of application of the previous result is for random field which
are known to be monochromatic, for an unknown eigenvalue; in these
circumstances our proposal for the derivation of $\sigma _{1}$ is equivalent
to the estimation of the dominating eigenvalue.\newline
\end{remark}

\section{Numerical results}

In this section, we validate our theoretical results with simulations.
Specifically, we consider a finite-rank random field
\begin{equation*}
Z(x,t) = \sum_{\ell=0}^{L} a_{\ell m}(t) Y_{\ell m}(x), \qquad x\in \mathbb{S%
}^2,\ t \in \mathbb{Z},
\end{equation*}
for some finite $L \ge 0 $. Each $\{a_{\ell m}(t),\ t \in \mathbb{Z}\} $ is
modelled as the increments of a fractional Brownian motion (fBm) with Hurst
parameter $H_\ell = d_\ell + 1/2, \ 0<d_\ell<1/2 $. More precisely, for each
fixed multipole $\ell $, we consider $(2\ell+1) $ independent zero-mean
Gaussian processes with covariance function
\begin{equation*}
C_\ell (\tau) = \frac{\gamma(0)}{2} \left( |\tau +1|^{2H_\ell} + |\tau -
1|^{2H_\ell} - 2 |\tau|^{2H_\ell} \right), \qquad \tau \in \mathbb{Z},
\end{equation*}
where $\gamma(0) = 4\pi \big/ \sum_{\ell=0}^L (2\ell +1) $ —so that
the field has unit variance. It is immediate to see that this is a
stationary covariance function with memory parameter $d_\ell $, and that the
assumptions used throughout the paper are satisfied.

We choose $L = 10 $; moreover, we consider the case $d_0 > 2 d_* - 1/2 $, by
choosing $d_\ell = 0.3 $ for every $\ell \in \{0,\dots,L\} $. In other
words, we are assuming to have the same degree of memory at every multipole,
and we expect $\Gamma(1, |\tau|) \approx |\tau|^{-0.4} $ as $\tau \to
+\infty $. The same behaviour is also expected for the autocovariance of the
geometric functionals computed at a fixed level $u $.

On the other hand, if we consider, for instance, the cointegrated residuals
\begin{equation*}
A(u_1,u_2;t) = A(u_1,t) - \frac{\phi(u_1)}{\phi(u_2)} A(u_2,t), \qquad t \in
\mathbb{Z},
\end{equation*}
for given levels $u_1, u_2 $, we expect the corresponding autocovariance
function to decay as $|\tau|^{-0.8} $: loosely speaking, these residuals
have "less memory" —this is indeed the definition of cointegration.
Figure \ref{fig:excursion} shows the excursions sets at given level $u=0.1$
for a single realization of the field at times $t=1,2,3,10.$ Figure~\ref%
{fig:path} (left) shows the evolution over time of the centered area at
levels $u_1 = -0.1 $ and $u_2 = 0.1 $, along with the corresponding
cointegrated residuals; the right panel shows the same for $u_1 = -0.5 $ and
$u_2 = 0.5 $.

\begin{figure}[ht!]
\center
\includegraphics[scale=.25]{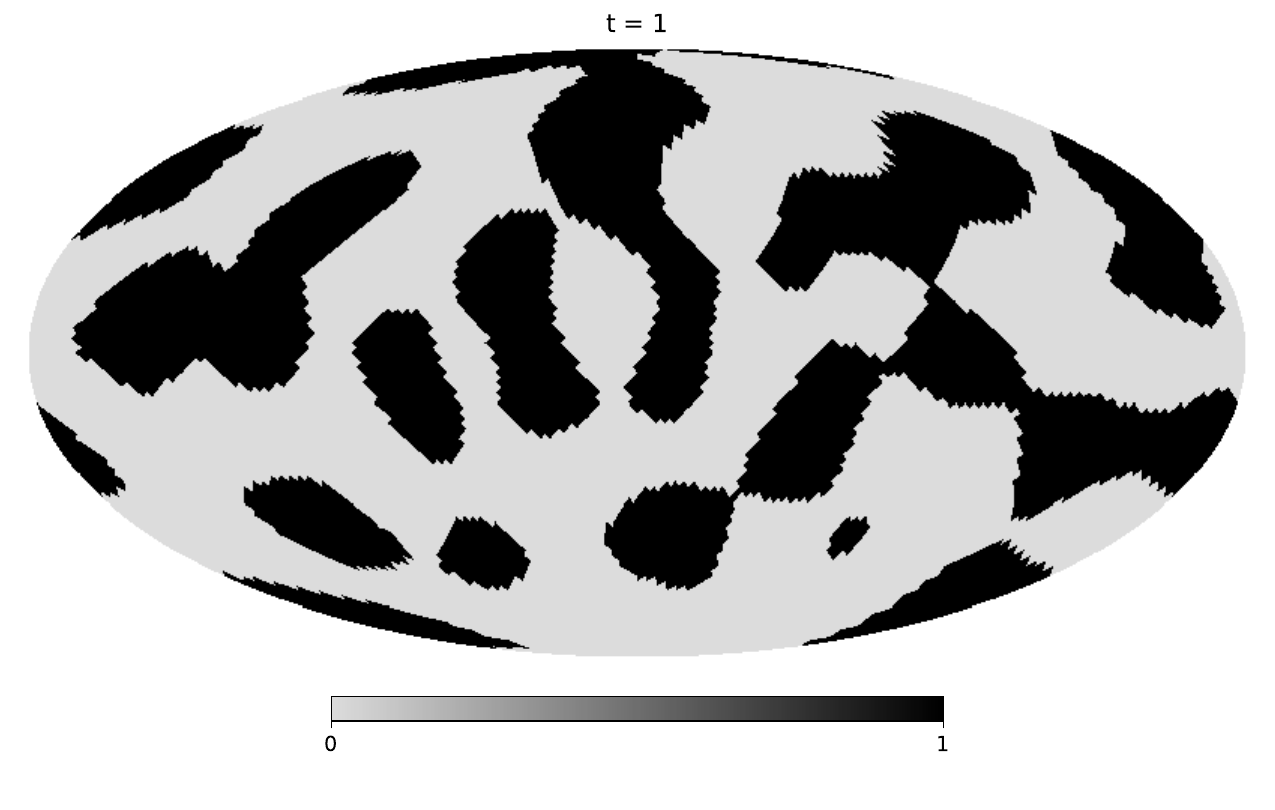} %
\includegraphics[scale=.25]{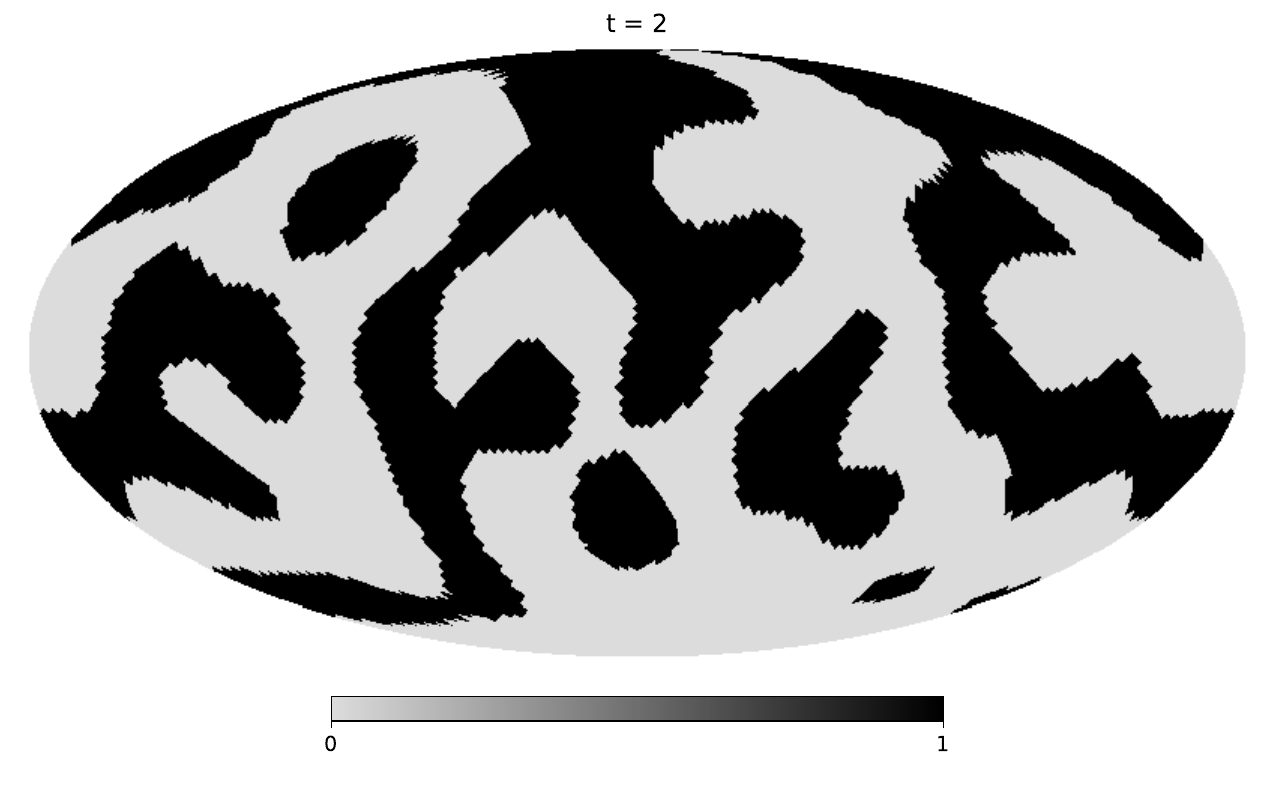} %
\includegraphics[scale=.25]{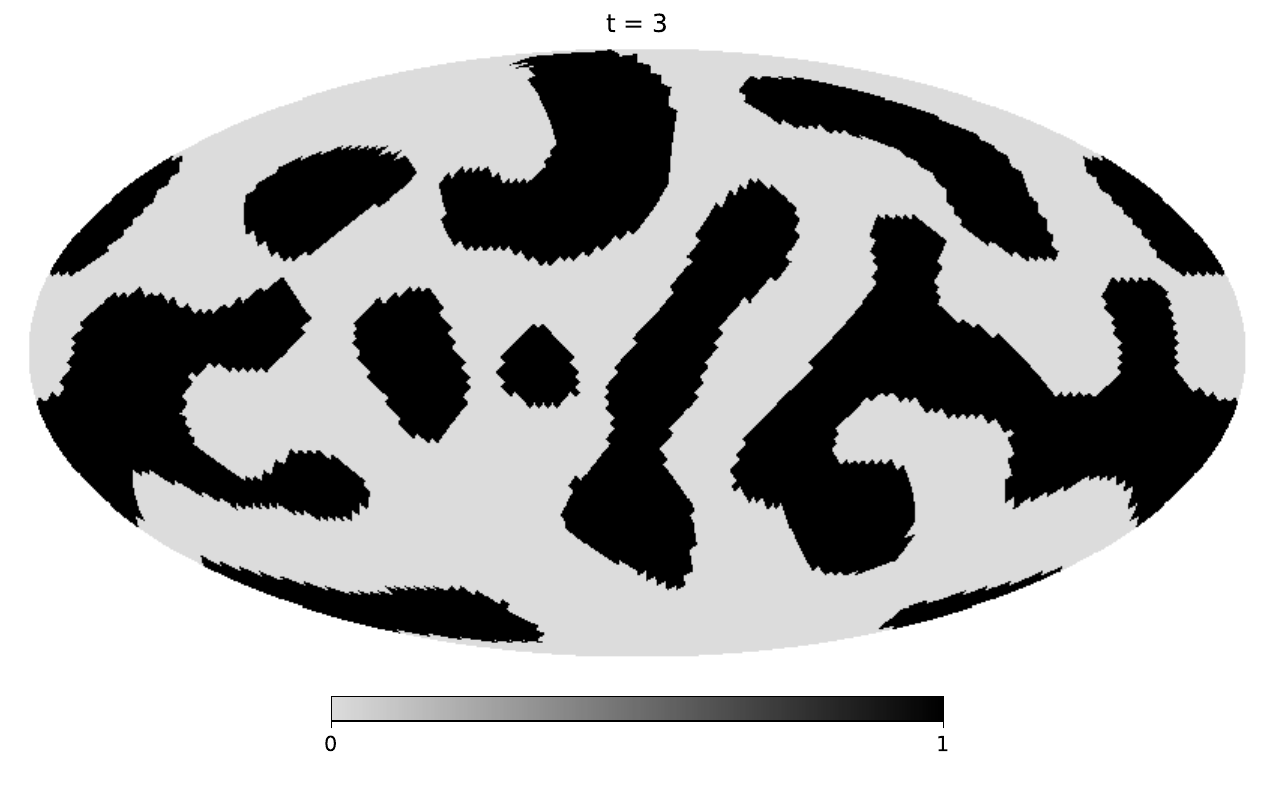} %
\includegraphics[scale=.25]{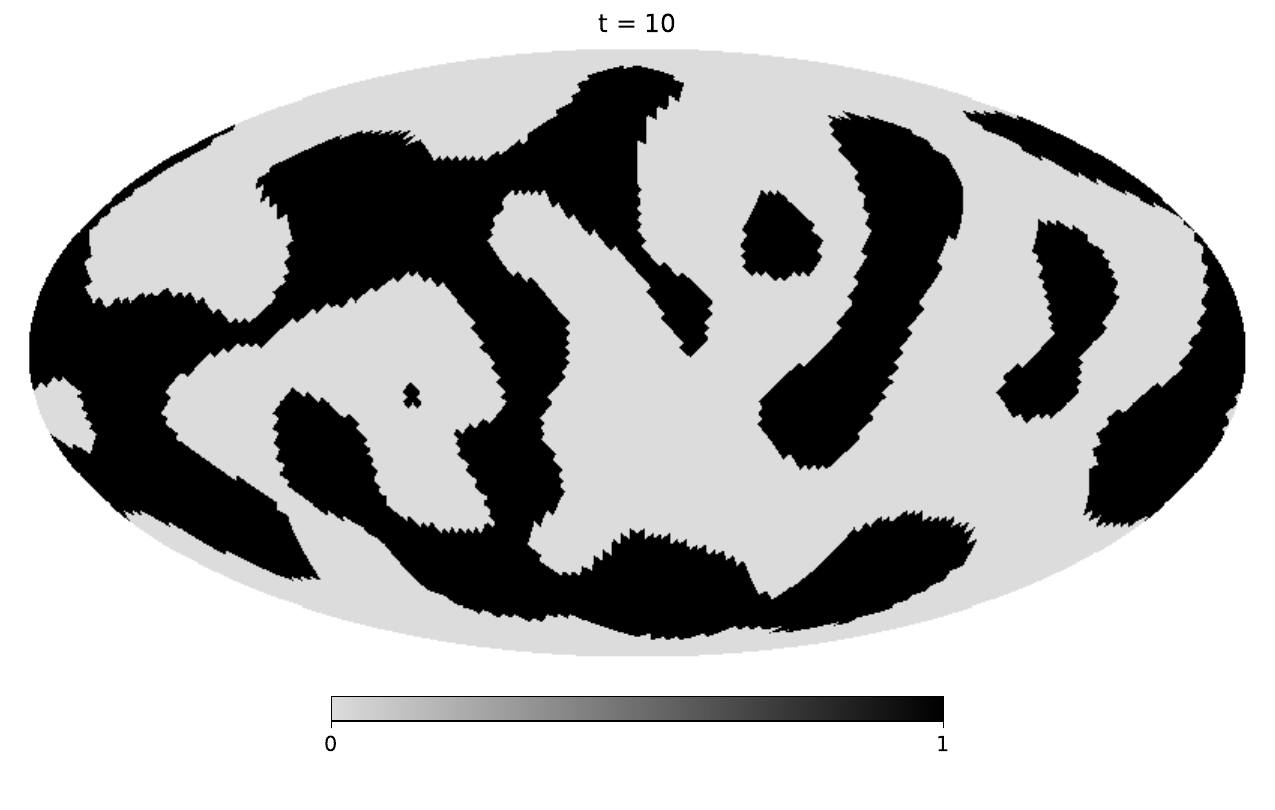}
\caption{Excursions sets at given level $u=0.1$ (black above $u$, grey below
$u$) for a single realization of the field at times $t=1,2,3,10.$}
\label{fig:excursion}
\end{figure}

\begin{figure}[ht!]
\center
\includegraphics[scale=.35]{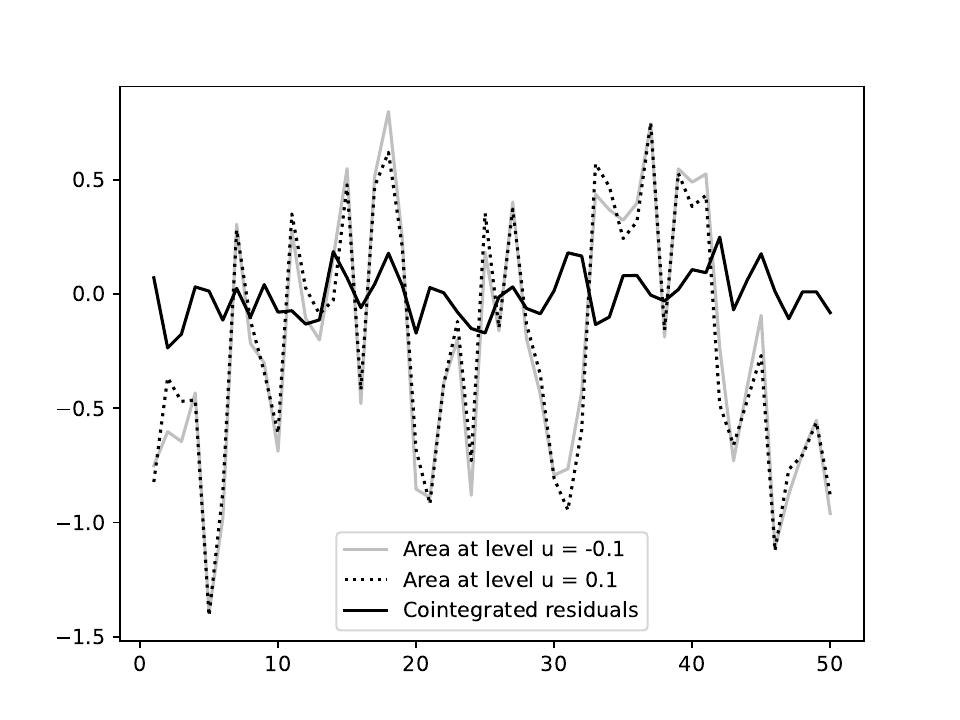} %
\includegraphics[scale=.35]{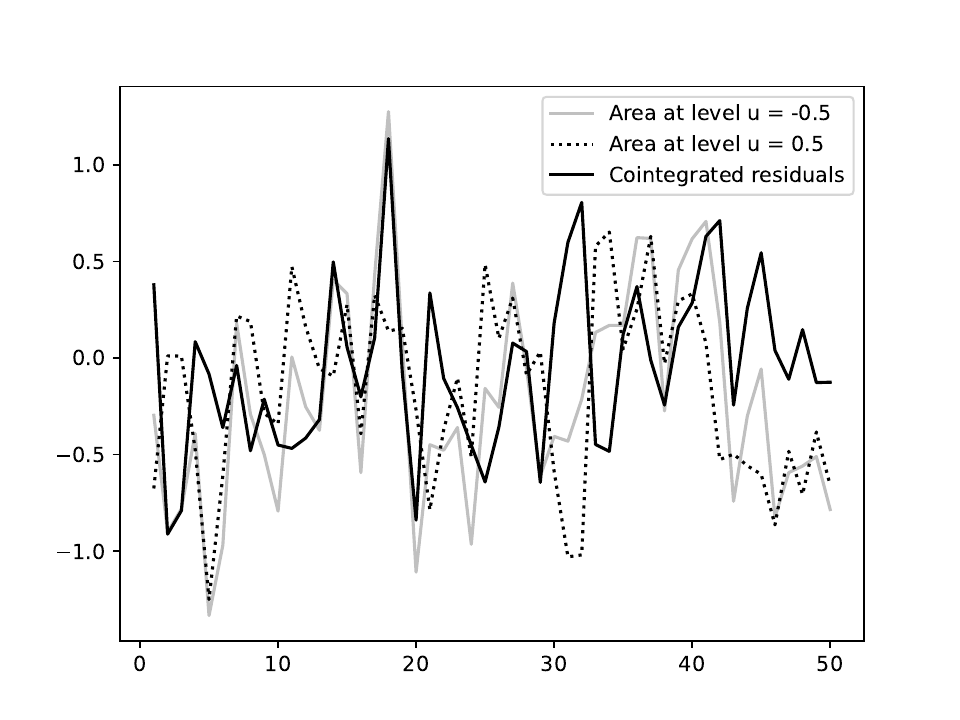}
\caption{Evolution over time of the centered area and cointegrated
residuals. Left: levels $u_1 = -0.1 $ and $u_2 = 0.1 $. Right: levels $u_1 =
-0.5 $ and $u_2 = 0.5 $.}
\label{fig:path}
\end{figure}

We validate these theoretical predictions by examining the empirical
autocovariance function
\begin{equation*}
\hat{\rho}(\tau) = \frac{1}{T - \tau} \sum_{t=1}^{T - \tau} X(t) X(t +
\tau), \qquad \tau = 1, \dots, T-1,
\end{equation*}
where
\begin{equation*}
X(t) =
\begin{cases}
Z(N, t), \\
A(u_1,t) - \mathbb{E}[A(u_1;t)], \\
A(u_2,t) - \mathbb{E}[A(u_2;t)], \\
A(u_1,u_2;t) - \mathbb{E}[A(u_1,u_2;t)].%
\end{cases}
\end{equation*}

Note that $\mathbb{E}[A(u_j;t)] = 4\pi \int_{u_j}^\infty \phi(u)\,du $, for $%
j = 1,2 $, and $N $ denotes the \emph{North Pole} (although any other $x \in
\mathbb{S}^2 $ can be chosen). There is a wide array of available estimators
for long memory parameters; for simplicity, we choose a log-regression on
the autocovariance function, i.e., we estimate the decay rate by solving the
following ordinary least squares problem
\begin{equation*}
\min_{\alpha, \beta \in \mathbb{R}}\sum_{\tau=1}^{q_T} \left ( \log |\hat{%
\rho}(\tau)| - \alpha - \beta \log \tau \right)^2,
\end{equation*}
with $q_T = \min \{\lfloor \log_{10} T\rfloor, T - 1\}$.

In practice, for each $\ell $, we simulate $(2\ell+1) $ i.i.d.\ fractional
Brownian motions with Hurst parameter $H_\ell = d_\ell + 1 = 0.8 $ on the
discrete grid $\{0, 1, \dots, T\} $, taking into account the normalisation
given by $\gamma(0) $. Then, we compute the increments of the differences
between the fBm realisations at time $t $ and those at time $t-1 $, for $t =
1, \dots, T $. We set $T = 1000 $.

From these simulated coefficients, we use the Python packages \texttt{healpy}
and \texttt{pynkowski} to obtain the field and the associated geometric
functionals for given threshold levels $u_1 = -0.1$ and $u_2=0.1$; finally,
we compute the realized empirical autocovariance function for each of these.
We repeat this routine $B = 1000 $ times and average the results to obtain
more stable estimates: as a final step, we perform an ordinary least squares
log-regression, as described above.

The results are shown in Table \ref{tab:ols_B}; we can observe that the
estimates obtained from the simulations match the theoretical decay rates
very well -- respectively, $|\tau|^{-0.4}$ and $|\tau|^{-0.8}$. Figure \ref%
{fig:ols_B} shows the observed autocovariances $\hat{\rho}(\tau)$, for lags $%
\tau =1,\dots, q_T,$ divided by $e^{\text{Intercept}}$ (see Table \ref%
{tab:ols_B} for the values of Intercept). This allows to clearly see the
decay of the autocovatiance functions up to scale effects, by showing the
difference between the considered functionals. The overall procedure is
repeated for $u_1=-0.5$ and $u_2=0.5$ -- see Table \ref{tab:ols_C} and
Figure \ref{fig:ols_C}.

\begin{table}[ht!]
\begin{center}
\begin{tabular}{lcccc}
\hline
& \textbf{Field} & \textbf{Area ($u_1 = -0.1$)} & \textbf{Area ($u_2 = 0.1$)}
& \textbf{Coint. resid.} \\ \hline
Intercept & -0.7182 & -2.2619 & -2.2627 & -6.8589 \\
$\log \tau$ & -0.4212 & -0.4122 & -0.4125 & -0.8120 \\ \hline
\end{tabular}%
\end{center}
\caption{Ordinary least squares solutions for $X(t)$ being respectively --
from left to right -- $Z(N, t)$ (field), $A(u_1,t) - \mathbb{E}[A(u_1;t)]$
(area at level $u_1=-0.1$), $A(u_2,t) - \mathbb{E}[A(u_2;t)]$ (area at level
$u_2=0.1$), and $A(u_1,u_2;t) - \mathbb{E}[A(u_1,u_2;t)]$ (cointegrated
residuals).}
\label{tab:ols_B}
\end{table}

\begin{table}[ht!]
\begin{center}
\begin{tabular}{lcccc}
\hline
& \textbf{Field} & \textbf{Area ($u_1 = -0.5$)} & \textbf{Area ($u_2 = 0.5$)}
& \textbf{Coint. resid.} \\ \hline
Intercept & -0.6921 & -2.4569 & -2.4690 & -3.8746 \\
$\log \tau$ & -0.4193 & -0.4331 & -0.4231 & -0.8220 \\ \hline
\end{tabular}%
\end{center}
\caption{Ordinary least squares solutions for $X(t)$ being respectively --
from left to right -- $Z(N, t)$ (field), $A(u_1,t) - \mathbb{E}[A(u_1;t)]$
(area at level $u_1=-0.5$), $A(u_2,t) - \mathbb{E}[A(u_2;t)]$ (area at level
$u_2=0.5$), and $A(u_1,u_2;t) - \mathbb{E}[A(u_1,u_2;t)]$ (cointegrated
residuals).}
\label{tab:ols_C}
\end{table}

\begin{figure}[ht!]
\center
\includegraphics[scale=.5]{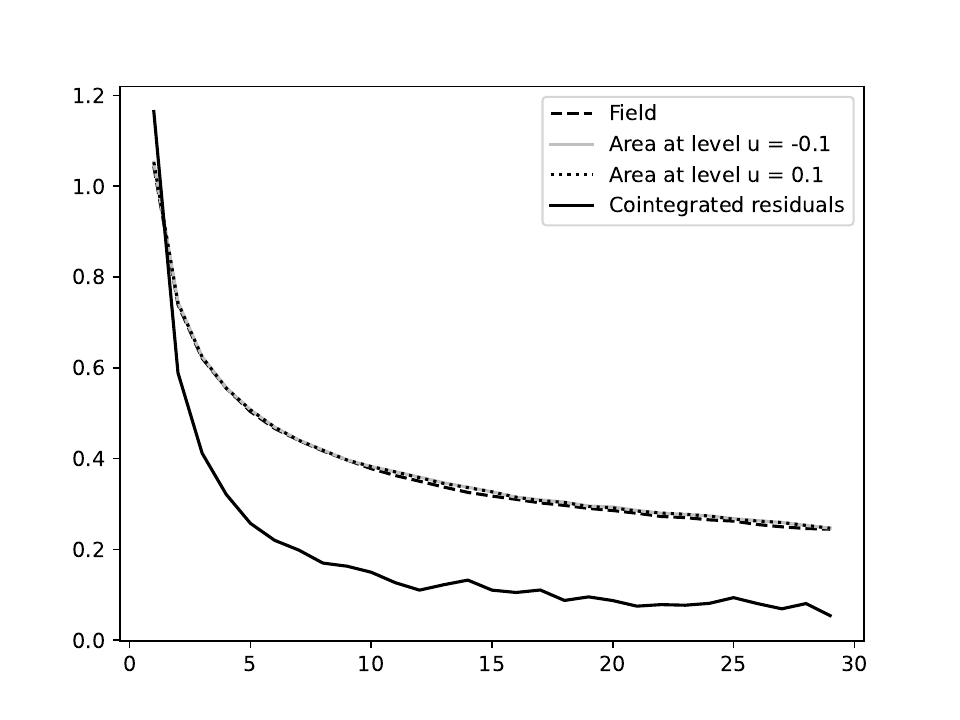}
\caption{Observed $\hat{\protect\rho}(\protect\tau)/ e^{\text{Intercept}}$,
for $\protect\tau =1,\dots, q_T$, for $X(t)$ being respectively $Z(N, t)$
(field), $A(u_1,t) - \mathbb{E}[A(u_1;t)]$ (area at level $u_1=-0.1$), $%
A(u_2,t) - \mathbb{E}[A(u_2;t)]$ (area at level $u_2=0.1$), and $%
A(u_1,u_2;t) - \mathbb{E}[A(u_1,u_2;t)]$ (cointegrated residuals). See Table
\protect\ref{tab:ols_B} for the values of Intercept.}
\label{fig:ols_B}
\end{figure}

\begin{figure}[ht!]
\center
\includegraphics[scale=.5]{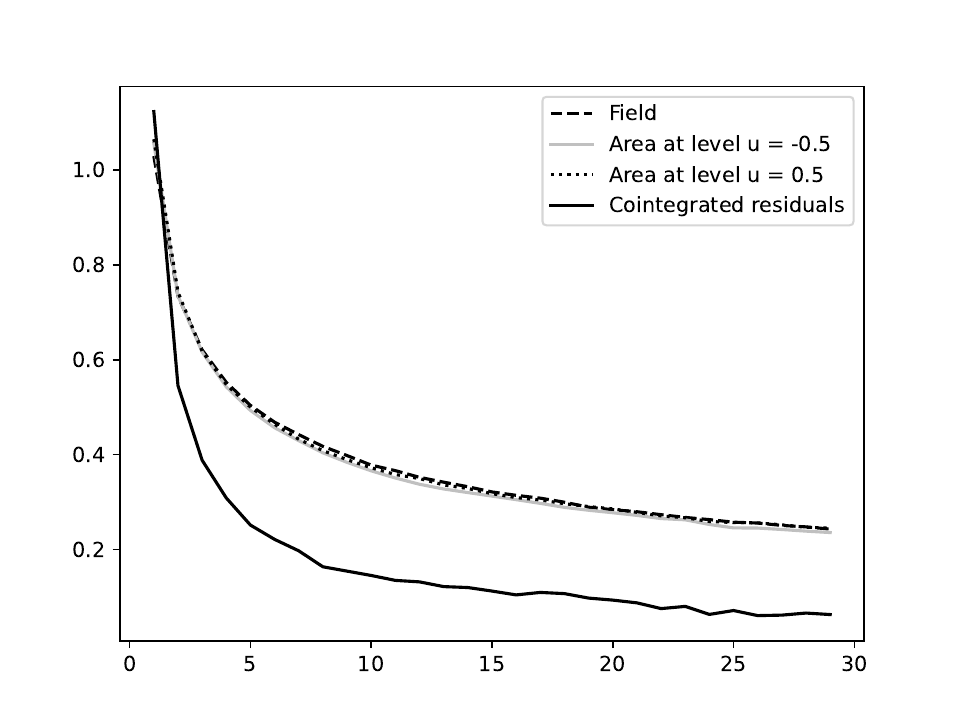}
\caption{Observed $\hat{\protect\rho}(\protect\tau)/ e^{\text{Intercept}}$,
for $\protect\tau =1,\dots, q_T$, for $X(t)$ being respectively $Z(N, t)$
(field), $A(u_1,t) - \mathbb{E}[A(u_1;t)]$ (area at level $u_1=-0.5$), $%
A(u_2,t) - \mathbb{E}[A(u_2;t)]$ (area at level $u_2=0.5$), and $%
A(u_1,u_2;t) - \mathbb{E}[A(u_1,u_2;t)]$ (cointegrated residuals). See Table
\protect\ref{tab:ols_C} for the values of Intercept.}
\label{fig:ols_C}
\end{figure}

\bigskip \bigskip \noindent \textbf{Acknowledgements. }The research leading
to this paper has been supported by PRIN project Grafia (CUP:
E53D23005530006) and PRIN Department of Excellence MatMod@Tov (CUP:
E83C23000330006). DM and AV are members of Indam/GNAMPA.


\begin{thebibliography}{99}
\bibitem{AdlerTaylor} Adler, Robert J.; Taylor, Jonathan E., \emph{Random
fields and geometry, }Springer Monogr. Math., Springer, New York, 2007,
xviii+448 pp.

\bibitem{Avarucci} Avarucci, Marco; Marinucci, Domenico, \emph{Polynomial
cointegration between stationary processes with long memory, }J.Time Ser.An.
28 (2007), no. 6, 923--942.

\bibitem{AzaisWschebor} Aza\"{\i}s, Jean-Marc; Wschebor, Mario, \emph{Level
sets and extrema of random processes and fields.} John Wiley \& Sons, Inc.,
Hoboken, NJ, 2009.

\bibitem{BGT} Bingham, Nicholas H.; Goldie, Charles M.; Teugels, Jozef L.,
\emph{Regular Variation, }Encyclopedia Math. Appl., 27, Cambridge University
Press, Cambridge, 1989.

\bibitem{AoS2001} Caponera, Alessia; Marinucci, Domenico, \emph{Asymptotics
for spherical functional autoregressions, }Ann. Statist. 49 (2021), no. 1,
346--369

\bibitem{Caponera2024} Caponera, Alessia; Rossi, Maurizia; Ruiz Medina, Mar%
\'{\i}a Dolores \emph{Sojourn functionals of time-dependent }$\chi ^{2}$%
\emph{-random fields on two-point homogeneous spaces}, arXiv:2403.17538

\bibitem{CaponeraSpharma} Caponera, Alessia, \emph{SPHARMA approximations
for stationary functional time series on the sphere}, Stat.Inf.Stoch.Proc.
24 (2021), no. 3, 609--634.

\bibitem{CaponeraLasso} Caponera, Alessia; Durastanti, Claudio; Vidotto,
Anna, \emph{LASSO estimation for spherical autoregressive processes, }%
Stoch.Proc.Appl. 137 (2021), 167--199.

\bibitem{CarliniSantucci} Carlini, Federico; Santucci de Magistris, Paolo,
\emph{On the identification of fractionally cointegrated VAR models with the
F(d) condition, }J.Bus.Econom.Statist. 37 (2019), no. 1, 134--146.

\bibitem{ChenHurvich2003} Chen, Willa W.; Hurvich, Clifford M., \emph{%
Semiparametric estimation of multivariate fractional cointegration, }%
J.Amer.Statist.Assoc. 98 (2003), no. 463, 629--642.

\bibitem{ChenHurvich2006} Chen, Willa W.; Hurvich, Clifford M., \emph{%
Semiparametric estimation of fractional cointegrating subspaces, }Ann.
Statist. 34 (2006), no. 6, 2939--2979.

\bibitem{DehlingTaqqu} Dehling, Herold; Taqqu, Murad S. \emph{The empirical
process of some long-range dependent sequences with an application to
U-statistics.} Ann. Statist. 17 (1989), no. 4, 1767--1783.

\bibitem{DobrushinMajor} Dobrushin, Roland L.; Major, P\'{e}ter, \emph{%
Non-central limit theorems for nonlinear functionals of Gaussian fields}. Z.
Wahrsch. Verw. Gebiete 50 (1979), no. 1, 27--52.

\bibitem{Iaco2019} Hualde, Javier; Iacone, Fabrizio, \emph{Fixed bandwidth
inference for fractional cointegration, }J.Time Ser.An. 40 (2019), no. 4,
544--572.

\bibitem{Johansen2019} Johansen, S\o ren; Nielsen, Morten \O rregaard, \emph{%
Nonstationary cointegration in the fractionally cointegrated VAR model, }%
J.Time Ser.An. 40 (2019), no. 4, 519--543.

\bibitem{LeonenkoMedina} Leonenko, Nikolai; Ruiz-Medina, Maria Dolores,
\emph{Sojourn functionals for spatiotemporal Gaussian random fields with
long memory}, J.Appl.Probab. 60 (2023), no. 1, 148--165.

\bibitem{MalPorcu} Malyarenko, Anatoliy; Porcu, Emilio, \emph{Multivariate
random fields evolving temporally over hyperbolic spaces, }J.Theoret.Probab.
37 (2024), no. 2, 975--1000

\bibitem{M2001} Marinucci, Domenico, \emph{Spectral regression for
cointegrated time series with long-memory innovations, }J.Time Ser.An. 21
(2000), no. 6, 685--705.

\bibitem{MaPeCUP} Marinucci, Domenico; Peccati, Giovanni, \emph{Random
fields on the sphere: Representation, limit theorems and cosmological
applications}, volume 389 of London Mathematical Society Lecture Note
Series. Cambridge University Press, Cambridge (2011)

\bibitem{MR2001} Marinucci, Domenico; Robinson, Peter M., \emph{%
Semiparametric fractional cointegration analysis,} J.Econometrics 105
(2001), no. 1, 225--247.

\bibitem{AAP2021} Marinucci, Domenico; Rossi, Maurizia; Vidotto, Anna, \emph{%
Non-universal fluctuations of the empirical measure for isotropic stationary
fields on} $\mathbb{S}^2\times \mathbb{R}$, Ann.Appl.Probab. 31 (2021), no.
5, 2311--2349.

\bibitem{AHL2024} Marinucci, Domenico; Rossi, Maurizia; Vidotto, Anna, \emph{%
Fluctuations of level curves for time-dependent spherical random fields, }%
Ann. H. Lebesgue 7 (2024), 583--620.

\bibitem{Ovalle} Ovalle-Mu\~{n}oz, Diana P.; Ruiz-Medina, Maria Dolores,
\emph{LRD spectral analysis of multifractional functional time series on
manifolds, }TEST 33 (2024), no. 2, 564--588.

\bibitem{PeccatiTaqqu} Peccati, Giovanni; Taqqu, Murad, S., \emph{Wiener
chaos: moments, cumulants and diagrams. A survey with computer
implementation.} Bocconi \& Springer Series, 1. Springer, Milan; Bocconi
University Press, Milan, 2011

\bibitem{Porcu2023} Porcu, Emilio; Feng, Samuel F.; Emery, Xavier; Peron,
Ana P., \emph{Rudin extension theorems on product spaces, turning bands, and
random fields on balls cross time, }Bernoulli 29 (2023), no. 2, 1464--1475.

\bibitem{Robinson1994} Robinson, Peter M., \emph{Semiparametric analysis of
long-memory time series,} Ann. Statist. 22 (1994), no. 1, 515--539.

\bibitem{RM2001} Robinson, Peter M.; Marinucci, Domenico, \emph{Narrow-band
analysis of nonstationary processes,} Ann. Statist. 29 (2001), no. 4,
947--986.

\bibitem{RM2003} Robinson, Peter M.; Marinucci, Domenico, \emph{%
Semiparametric frequency domain analysis of fractional cointegration,} in
Time Series with Long Memory, P.M.Robinson ed., Adv. Texts Econometrics,
Oxford University Press, Oxford, 2003, 334--373.

\bibitem{RobinsonYajima} Robinson, Peter M.; Yajima, Yoshihiro, \emph{%
Determination of cointegrating rank in fractional systems.} J. Econometrics
106 (2002), no. 2, 217--241

\bibitem{Taqqu1975} Taqqu, Murad S., \emph{Weak convergence to fractional
Brownian motion and to the Rosenblatt process.} Z.
Wahrscheinlichkeitstheorie und Verw. Gebiete 31 (1974/75), 287--302.

\bibitem{Taqqu1979} Taqqu, Murad S., \emph{Convergence of integrated
processes of arbitrary Hermite rank.} Z. Wahrsch. Verw. Gebiete 50 (1979),
no. 1, 53--83.

\bibitem{Velasco2003} Velasco, Carlos, \emph{Gaussian semi-parametric
estimation of fractional cointegration, }J.Time Ser.An. 24 (2003), no. 3,
345--378.

\bibitem{Robinson2019} Zhang, Rongmao; Robinson, Peter; Yao, Qiwei, \emph{%
Identifying cointegration by eigenanalysis, }J. Amer. Statist. Assoc. 114
(2019), no. 526, 916--927.
\end{thebibliography}
\end{document}